\documentclass[12pt]{amsart}
\usepackage[margin=1in]{geometry}
\usepackage{latexsym,amscd,amssymb,epsfig} 
\usepackage[dvips]{color}
\usepackage{dsfont}
\usepackage{graphicx}
\usepackage{amsthm}
\usepackage{ytableau}
\usepackage{amsmath,amsfonts}
\usepackage{amssymb,verbatim}
\usepackage{float}
\usepackage{tikz}
\usepackage{enumerate}
\usepackage{multicol}
\usepackage{caption}
\usepackage{subcaption}

\newtheorem{thm}{Theorem}[section]
\newtheorem{lemma}[thm]{Lemma}

\newtheorem{proposition}[thm]{Proposition}
\newtheorem{cor}[thm]{Corollary}

\newtheorem*{thm*}{Theorem}
\newtheorem*{lemma*}{Lemma}
\newtheorem*{prop*}{Proposition}
\newtheorem*{cor*}{Corollary}
\newtheorem*{conj*}{Conjecture}

\theoremstyle{definition}
\newtheorem{definition}[thm]{Definition}
\newtheorem{ex}[thm]{Example}

\theoremstyle{remark}
\newtheorem{rmk}{Remark}

\newcommand{\rr}{\mathbb{R}}

\def\integers{{\mathbb Z}}

\usepackage[backend=bibtex]{biblatex}
\newcommand{\ind}{\mbox{$\perp \kern-5.5pt \perp$}}

\begin{document}

\title[Relations in doubly laced crystal graphs via discrete Morse theory]{Relations in doubly laced crystal graphs via discrete Morse theory}
\thanks{ The author was partially supported by NSF grant DMS-1500987.}




\author{Molly Lynch}
\address{North Carolina State University}
\email{melynch4@ncsu.edu}

\maketitle

\begin{abstract}
We study the combinatorics of crystal graphs given by highest weight representations of types $A_{n}, B_{n}, C_{n}$, and $D_{n}$, uncovering new relations that exist among crystal operators. Much structure in these graphs has been revealed by local relations given by Stembridge and Sternberg. However, there exist relations among crystal operators that are not implied by Stembridge or Sternberg relations. Viewing crystal graphs as edge colored posets, we use poset topology to study them. Using the lexicographic discrete Morse functions of Babson and Hersh, we relate the M\"obius function of a given interval in a crystal poset of simply laced or doubly laced type to the types of relations that can occur among crystal operators within this interval.

For a crystal of a highest weight representation of finite classical Cartan type, we show that whenever there exists an interval whose M\"obius function is not equal to -1, 0, or 1, there must be a relation among crystal operators within this interval not implied by Stembridge or Sternberg relations. As an example of an application, this yields  relations among crystal operators in type $C_{n}$ that were not previously known. Additionally, by studying the structure of Sternberg relations in the doubly laced case, we prove that crystals of highest weight representations of types $B_{2}$ and $C_{2}$ are not lattices.  
\end{abstract}

\section{Introduction}
\label{intro}
 Crystal graphs are edge colored, directed graphs that give information regarding the representations of Lie algebras. In this paper, we study crystals given by highest weight representations of finite classical Cartan type, namely crystals of types $A_{n}, B_{n}, C_{n}$, and $D_{n}$. 
We are interested in understanding relations that occur among crystal operators. In \cite{Stembridge}, Stembridge gives a local characterization of crystals of highest weight representations in the simply laced case. The axioms stated in his paper imply a list of local relations that exist among crystal operators. These local relations determine much of the structure of the crystal graph. He shows these relations also hold for crystals of doubly laced type. In \cite{Sternberg}, Sternberg proves that additional local relations exist among crystal operators in the doubly laced case. In spite of this, in the simply laced case, there do exist relations among crystal operators that are not implied by Stembridge relations as seen in \cite{HershLenart}.
 
The crystal graphs which we are interested in, namely crystals arising from highest weight representations of finite type, are equipped with a natural partial ordering (see Section 2 for background information on partially ordered sets). This partial order is given by covering relations, denoted $\lessdot$, as follows: we say that $x \lessdot y$ whenever $y = f_{i}(x)$, where $f_{i}$ is a so-called crystal operator. We color each of these covering relations with $i$, giving the crystal the structure of an edge colored poset. 
 
The question of what types of relations can exist among crystal operators has been previously studied by Hersh and Lenart in \cite{HershLenart} in the simply laced case. They show that for arbitrary intervals in crystals of simply laced type, there exist relations among crystal operators not implied by Stembridge relations. More generally, Hersh and Lenart prove that whenever there is an interval $[u,v]$ in a crystal of finite, simply-laced type with the M\"obius function $\mu(u,v) \notin \{-1,0,1\}$, then within $[u,v]$ there exists a relation among crystal operators not implied by Stembridge relations. 

We prove the analogue of this result for crystals of finite, doubly laced type, which was not previously known, using a tool developed in \cite{HershBabson} known as lexicographic discrete Morse functions. In doing so, we give a new proof of the result in the simply laced case. More specifically, we show that if we have an interval $[u,v]$ in a crystal poset of finite classical Cartan type such that all relations among crystal operators are implied by Stembridge or Sternberg relations, then the M\"obius function of this interval must be equal to $-1, 0,$ or $1$. We do so by constructing a discrete Morse function on the order complex, $\Delta(u,v)$, with at most one critical cell. We give a procedure for determining if $[u,v]$ has a critical cell, and finding this cell when it exists. If the discrete Morse function has exactly one critical cell, this results in the M\"obius function of the interval equalling $\pm 1$, else the M\"obius function equals 0.

Danilov, Karzanov, and Koshevoy have studied these crystal posets in case when $n = 2$ in \cite{Danilov, DKK}. They show that crystals of highest weight representations of type $A_{2}$ are in fact lattices. In the present paper, by studying the structure of the Sternberg relations, we prove that crystals of highest weight representations of types $B_{2}$ and $C_{2}$ are not lattices. Additionally, using SAGE, we search for intervals in crystal posets with M\"obius function not equal to -1,0, or 1. By doing so, we find new relations among crystal operators in crystals of type $C_{3}$.  

Our main results, Corollaries \ref{simply laced mobius} and \ref{doubly laced mobius} consider intervals in crystal posets where all relations among crystal operators are implied by Stembridge or Sternberg relations. Now let us describe and illustrate the main ideas of this paper through an example. 

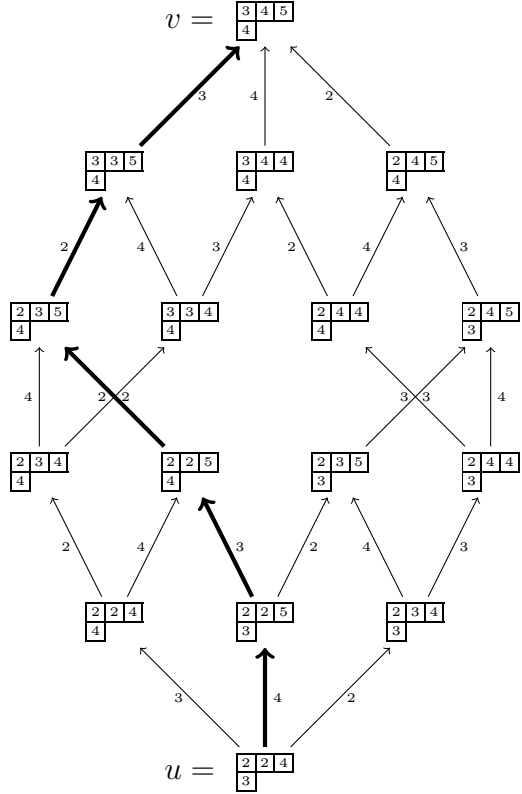
\begin{figure}[h] 
\ytableausetup{boxsize=1em}
\fontsize{4.8pt}{6.7}
\centering
	\begin{tikzpicture}
    \node (u) at (-1,0) {\normalsize $u = $};
    \node (a) at (0,0) {$\begin{ytableau} 2 & 2 & 4 \\ 3 \end{ytableau}$};
    \node (b) at (-2,2) {$\begin{ytableau} 2 & 2 & 4 \\ 4 \end{ytableau}$};
    \node (c) at (0,2)  {$\begin{ytableau} 2 & 2 & 5 \\ 3 \end{ytableau}$};
    \node (d) at (2,2)  {$\begin{ytableau} 2 & 3 & 4 \\ 3 \end{ytableau}$};
    \node (e) at (-3,4)  {$\begin{ytableau} 2 & 3 & 4 \\ 4 \end{ytableau}$};
    \node (f) at (-1,4) {$\begin{ytableau} 2 & 2 & 5 \\ 4 \end{ytableau}$};
    \node (g) at (1,4) {$\begin{ytableau} 2 & 3 & 5 \\ 3 \end{ytableau}$};  
    \node (h) at (3,4) {$\begin{ytableau} 2 & 4 & 4 \\ 3 \end{ytableau}$};     
    \node (i) at (-3,6) {$\begin{ytableau} 2 & 3 & 5 \\ 4 \end{ytableau}$};
    \node (j) at (-1,6) {$\begin{ytableau} 3 & 3 & 4 \\ 4 \end{ytableau}$};
    \node (k) at (1,6) {$\begin{ytableau} 2 & 4 & 4 \\ 4 \end{ytableau}$};
    \node (l) at (3,6) {$\begin{ytableau} 2 & 4 & 5 \\ 3 \end{ytableau}$};
    \node (m) at (-2,8) {$\begin{ytableau} 3 & 3 & 5 \\ 4 \end{ytableau}$};
    \node (n) at (0,8) {$\begin{ytableau} 3 & 4 & 4 \\ 4 \end{ytableau}$};
    \node (o) at (2,8) {$\begin{ytableau} 2 & 4 & 5 \\ 4 \end{ytableau}$};
    \node (p) at (0,10) {$\begin{ytableau} 3 & 4 & 5 \\ 4 \end{ytableau}$};
    \node (v) at (-1,10) {\normalsize $v = $};

    \path [draw=green, ->] (a) edge node[left] {$3$} (b);
    \path [ultra thick, draw=cyan, ->](a) edge node[right] {$4$} (c);
    \path [draw=red, ->](a) edge node[right] {$2$} (d);
    \path [draw=red, ->](b) edge node[left] {$2$} (e);
    \path [draw=cyan, ->](b) edge node[left] {$4$} (f);
    \path [ultra thick, draw=green, ->](c) edge node[right] {$3$} (f);
    \path [draw=red, ->](c) edge node[right] {$2$} (g);
    \path [draw=cyan, ->](d) edge node[left] {$4$} (g); 
    \path [draw=green, ->](d) edge node[right] {$3$} (h); 
    \path [draw=cyan, ->](e) edge node[left] {$4$} (i); 
    \path [draw=red, ->](e) edge node[right] {$2$} (j); 
    \path [ultra thick, draw=red, ->](f) edge node[left] {$2$} (i); 
    \path [draw=green, ->](g) edge node[right] {$3$} (l); 
    \path [draw=green, ->](h) edge node[left] {$3$} (k); 
    \path [draw=cyan, ->](h) edge node[right] {$4$} (l); 
    \path [ultra thick, draw=red, ->](i) edge node[left] {$2$} (m);
    \path [draw= cyan, ->](j) edge node[left] {$4$} (m);
    \path [draw=green, ->] (j) edge node[left] {$3$} (n);
    \path [draw=red, ->] (k) edge node[left] {$2$} (n);
    \path [draw=cyan, ->] (k) edge node[left] {$4$} (o);
    \path [draw=green, ->] (l) edge node[right] {$3$} (o);
    \path [ultra thick, draw=green, ->] (m) edge node[right] {$3$} (p);
    \path [draw=cyan, ->] (n) edge node[left] {$4$} (p);
    \path [draw=red, ->] (o) edge node[left] {$2$} (p);
   \end{tikzpicture}
\caption{Subposet of type $A_{4}$ crystal with highest weight $\lambda  = (3,1,0,0)$} \label{ex1}
\end{figure}

The interval $[u,v]$ in Figure \ref{ex1} is a subposet of the crystal of type $A_{4}$ with highest weight $(3,1,0,0)$. We order the saturated chains in our interval according to lexicographic order on their edge label sequences as we travel up each chain from $u$ to $v$. The critical cells in our lexicographic discrete Morse function come from so-called fully covered saturated chains in the interval. Informally, we have a fully covered saturated chain $C$ from $u$ to $v$ when each rank along $C$, excluding that of $u$ and $v$, is covered by a ``minimal skipped interval". Roughly speaking, we have a skipped interval from $u'$ to $v'$ consisting of all elements strictly between $u'$ and $v'$ along $C$ if there is a lexicographically earlier chain $C'$ from $u'$ to $v'$. If there are no strictly smaller skipped intervals between $u'$ and $v'$ then we have a minimal skipped interval. The technique we are using is a generalization of a lexicographic shelling. We differ from lexicographic shellings as we allow our minimal skipped intervals to cover more than one rank. 

Consider the chain in bold in our example. This chain has label sequence $(4,3,2,2,3)$. Let us explain how this saturated chain is fully covered by looking at it's minimal skipped intervals. For our first minimal skipped interval, instead of traveling up this chain via the edges labeled $4$ and $3$, we could have traveled up the lexicographically earlier segment via the edges labeled $3$ and then $4$. Next, instead of traveling along the edges labeled by the sequence $(3,2,2,3)$, we could have traveled up the lexicographically earlier segment labeled $(2,3,3,2)$. These two minimal skipped intervals cover all proper ranks of our interval and so the chain with label sequence $(4,3,2,2,3)$ is fully covered. This is the only fully covered saturated chain within $[u,v]$. As having a fully covered saturated chain gives rise to a critical cell in our discrete Morse function, we are able to deduce that the M\"obius function, $\mu(u,v)$, of our interval is $-1$. 

We give an algorithm for finding a fully covered saturated chain from $u$ to $v$ in crystals of highest weight representations of finite classical Cartan type where all relations are implied by Stembridge or Sternberg relations when one exists. In the process, we prove that there is at most one such fully covered saturated chain from $u$ to $v$. We note that when a fully covered saturated chain exists, it is not always the lexicographically last chain, though often it is. For such an instance, see Example \ref{type D example}.

We give background information on crystals, partially ordered sets, and discrete Morse functions in Section 2. In Section 3, we point out some immediate consequences of the Stembridge axioms for the simply laced and doubly laced cases. We use the structure of the degree five Sternberg relation to prove that crystals of highest weight representations of types $B_{2}$ and $C_{2}$ are not lattices. In Section 4, we construct lexicographic discrete Morse functions for intervals in crystals of highest weight representations of finite classical Cartan type where all relations among crystal operators are implied by Stembridge or Sternberg relations. This construction allows us to prove the main result, namely that if there is an interval in a crystal of finite classical Cartan type with M\"obius function not equal to $-1, 0,$ or $1$, then there exists a relation among crystal operators within that interval not implied by Stembridge or Sternberg relations. Finally, we give a concrete application of the main result in Section 5 demonstrating how it can lead to the discovery of new relations among crystal operators via computer search. Specifically, we find new relations among crystal operators in crystals of type $C_{3}$. 
 
\textbf{Acknowledgements:} This work is part of the author's PhD research. The author thanks Patricia Hersh for many helpful discussions regarding this project. Additionally, the author thanks Nicolas Thi\'ery for help with SAGE, and Nathan Reading and Ricky Liu for helpful comments and suggestions.

\section{Background and terminology}
\subsection{Crystal bases} 

Crystals bases are combinatorial structures that give information regarding representations of Lie algebras. Each crystal has an associated root system $\Phi$ with index set $I$ and weight lattice $\Lambda$. Let $\{ \alpha_{i} \}_{i \in I}$ be the set of simple roots and $\Lambda^{+}$ be the set of dominant integral weights.
 We will focus on the the root systems of finite type. For more background on root systems, see \cite{Humphreys}.\
\begin{definition}
For a fixed root system $\Phi$ with index set $I$ and weight lattice $\Lambda$, a \textbf{Kashiwara crystal} (crystal for short) of type $\Phi$ is a nonempty set $\mathcal{B}$ together with maps
\begin{subequations}
\begin{align}
        e_{i}, f_{i} &: \mathcal{B} \rightarrow \mathcal{B} \sqcup \{0\},\\
        \varepsilon_{i}, \varphi_{i}&: \mathcal{B} \rightarrow \integers \sqcup \{-\infty\}, \\
        \text{wt} &: \mathcal{B} \rightarrow \Lambda,
\end{align}
\end{subequations}
where $i \in I$ and $0 \notin \mathcal{B}$ is an auxillary element satisfying the following: 
\begin{enumerate}
	\item[(A1)] If $x,y \in \mathcal{B}$ then $e_{i}(y) = x$ if and only if $f_{i}(x) = y$, and in this case we assume
		\begin{equation*} 
			\text{wt}(x) = \text{wt}(y) + \alpha_{i}, \quad \varepsilon_{i}(x) = \varepsilon_{i}(y) -1, \quad \varphi_{i}(x) = \varphi_{i}(y) + 1
			\end{equation*}
	\item[(A2)] We require that 
		\begin{equation*}
		\varphi_{i}(x) = \langle \text{wt}(x), \alpha_{i}^{\vee} \rangle + \varepsilon_{i}(x)
		\end{equation*} 
		for all $x \in \mathcal{B}$ and $i \in I$. In particular, if $\varphi(x) = -\infty$, then $\varepsilon_{i}(x) = -\infty$ as well. If $\varphi_{i}(x) = -\infty$ then we require $e_{i}(x) = f_{i}(x) = 0.$
\end{enumerate}
\end{definition} 
The map wt is the \textit{weight map}. The operators $e_{i}, f_{i}$ are called \textit{Kashiwara or crystal operators}, and the maps $\varphi_{i}, \varepsilon_{i}$ are called \textit{string lengths}. 

We will only be referring to crystals of highest weight representations of finite type and our main results will hold for crystals of highest weight representations of the classical Cartan algebras $A_{n}, B_{n}, C_{n},$ and $D_{n}$. For a dominant integral weight $\lambda \in \Lambda^{+}$, we let $\mathcal{B} = \mathcal{B}_{\lambda}$ denote the crystal of the irreducible representation $V(\lambda)$ of highest weight $\lambda$.

Given any crystal $\mathcal{B}$, we can associate to it a \textit{crystal graph}. 

 \begin{definition}
A \textit{crystal graph} of some crystal $\mathcal{B}$ is a directed, edge colored (with colors $i \in I$) graph with vertices in $\mathcal{B}$ satisfying the following:
\begin{enumerate}
\item all monochromatic directed paths have finite length,
\item for every vertex $x \in \mathcal{B}$ and $i \in I$, there is at most one edge $z \xrightarrow{i} x$, and dually, at most one edge $x \xrightarrow{i} y$. Here, we say that $z=e_{i}(x)$ and $y=f_{i}(x)$.
\end{enumerate}
\end{definition}





For crystals of finite, classical Cartan type there is a combinatorial model, developed in \cite{KashiwaraNakashima}, where the vertices of the crystal graph are represented by tableaux. While this combinatorial realization is not needed for the proofs, it may help readers unfamiliar with crystals to be able to compute and understand examples.

For crystals of type $A_{n}$, a dominant weight $\lambda$, can be viewed as a partition, i.e. $\lambda = (\lambda_{1}, \lambda_{2}, ... , \lambda_{n})$ where $\lambda_{1} \geq \lambda_{2} \geq \cdots \geq \lambda_{n}.$ We can represent vertices of our crystal graph by semistandard Young tableaux. A semistandard Young tableaux of shape $\lambda$ is a filling of a partition of shape $\lambda$ where entries across rows read from left to right are weakly increasing, and strictly increasing along columns read from top to bottom. Given a tableau $T$ in our crystal graph of type $A_{n}$, we can read off the weight of $T$ as follows: \[\text{wt}(T) = (\gamma_{1}, \gamma_{2}, ... ,\gamma_{n+1}),\] where $\gamma_{i}$ is the number of occurrences of $i$ in $T$. For crystals of types $B_{n}, C_{n},$ and $D_{n}$, vertices can be represented by Kashiwara-Nakashima tableaux. For a description of these tableaux see e.g \cite{BumpSchilling, KashiwaraNakashima}.

We describe the action of the crystal operators $f_{i}$ and $e_{i}$ on a given tableau via the signature rule. To begin, we describe the signature rule for crystals of type $A_{n}$. Let $x$ be a vertex of $\mathcal{B}$, a crystal of type $A_{n}$. 
\begin{definition}
The \textbf{reading word} of $x$, denoted $r(x)$, is the word obtained by reading each column from bottom to top and reading columns from left to right.
\end{definition}
\begin{definition}
The \textbf{$i$-signature} of $x$ is the subword of $r(x)$ consisting of only the letters $i$ and $i+1$.
\end{definition}
We replace each appearance of $i$ in the $i$-signature with the symbol $+$ and each appearance of $i+1$ with the symbol $-$. Then we repeatedly remove any adjacent pairs of $(-+)$ as long as this is possible. In the end, we are left with the reduced $i$-signature of $x$, denoted $\rho_{i}(x)$, which is of the form:
\begin{equation*}
\rho_{i}(x) = \underbrace{+ + \cdots +}_{a} \underbrace{- - \cdots -}_{b}.
\end{equation*}
Then, if $a > 0$ we obtain $f_{i}(x)$ from $x$ by changing the $i$ in $x$ that corresponds to the rightmost $+$ in $\rho_{i}(x)$ to an $i+1$. If $a=0$ then $f_{i}(x) = 0$. Similarly, if $b > 0$ then $e_{i}(x)$ is obtained by changing the $i+1$ in $x$ that corresponds to the leftmost $i+1$ in $\rho_{i}(x)$ to an $i$. If $b=0$, then $e_{i}(x) = 0$. For types $B_{n}, C_{n},$ and $D_{n}$, the signature rule for the applications of $f_{i}$ and $e_{i}$ are similar. For details see e.g. \cite{BumpSchilling, HongKang}.

\subsection{Stembridge axioms and Sternberg relations}
Given any integrable highest weight representation of a symmetrizable quantum Kac Moody algebra, there is a crystal. However, not all crystals arise in this way. Stembridge was interested in finding a simple set of axioms that characterize such crystals. In \cite{Stembridge}, this is done for simply laced types. In addition to characterizing crystals arising from integrable highest weight representations of simply laced type, Stembridge shows that these axioms also hold in the doubly laced case. However, they do not provide a complete characterization. In \cite{Sternberg}, Sternberg proves that for crystals of doubly laced type, additional local relations exist beyond those implied by the Stembridge axioms. For a complete characterization of doubly laced crystals see \cite{Danilov, Tsuchioka}. However, for this paper, we only need the Stembridge axioms and Sternberg relations given in \cite{Sternberg}. Now, we introduce some notation and the axioms as seen in \cite{Stembridge}.

Throughout this section, we will let $A = [a_{ij}]_{i,j \in I}$ be the Cartan matrix of a Kac Moody algebra $\mathfrak{g}$, where $I$ is some finite index set. We recall the following definition from \cite{Stembridge}:
\begin{definition} We say that an edge colored directed graph, $X$, is \textit{A-regular} if the axioms $(S1)-(S6)$ and $(S5')-(S6')$ are satsified:
\begin{enumerate}
\item[(S1)] All monochromatic paths in $X$ have finite length.
\item[(S2)] For any vertex $x \in X$, there is at most one incoming edge, $z \xrightarrow{i} x$, colored $i$ and at most one outgoing edge, $x \xrightarrow{i} y$ colored $i$. Here, we say that $z = e_{i}(x)$ and $y = f_{i}(x)$, where $e_{i}$ and $f_{i}$ are \textit{crystal operators}.
\end{enumerate}
We can define the \textit{i-string through x} to be:
\begin{equation*}
f_{i}^{-d}(x) \rightarrow \cdots \rightarrow f_{i}^{-1}(x) \rightarrow x \rightarrow f_{i}(x) \rightarrow \cdots \rightarrow f_{i}^{r}(x).
\end{equation*} 
We can then define the \textit{i-rise of x} to be $\varepsilon_{i}(x) := r$ and the \textit{i-depth of x} to be $\delta_{i}(x) := -d$. To measure the effect of the crystal operators $e_{i}$ and $f_{i}$ on the $j$-rise and $j$-depth of each vertex, we define the \textit{difference operators} $\Delta_{i}$ to be:
\begin{equation*}
\Delta_{i}\delta_{j}(x) = \delta_{j}(e_{i}(x)) - \delta_{j}(x), \quad  \Delta_{i}\varepsilon_{j}(x) = \varepsilon_{j}(e_{i}(x)) - \varepsilon_{j}(x),
\end{equation*}
whenever $e_{i}(x)$ is defined, and
\begin{equation*}
\nabla_{i}\delta_{j}(x) = \delta_{j}(x) - \delta_{j}(f_{i}(x)), \quad \nabla_{i}\varepsilon_{j}(x) = \varepsilon_{j}(x) - \varepsilon_{j}(f_{i}(x)),
\end{equation*}
whenever $f_{i}(x)$ is defined.
\begin{enumerate}
\item[(S3)] For a fixed $x \in X$ and $i,j \in I$ such that $e_{i}(x)$ is defined, we require $\Delta_{i}\delta_{j}(x) + \Delta_{i}\varepsilon_{j}(x) = a_{ij}$,
\item[(S4)] For a fixed $x \in X$ and $i,j \in I$ such that $e_{i}(x)$ is defined, we require $\Delta_{i}\delta_{j}(x) \leq 0$ and $\Delta_{i}\varepsilon_{j}(x) \leq 0$.
\item[(S5)] For a fixed $x \in X$ such that $e_{i}(x)$ and $e_{j}(x)$ are both defined, we require that $\Delta_{i}\delta_{j}(x) = 0$ implies $e_{i}e_{j}(x) = e_{j}e_{i}(x)$ and $\nabla_{j}\varepsilon_{i}(y) = 0$ where $y = e_{i}e_{j}(x) = e_{j}e_{i}(x)$.
\item[(S6)] For a fixed $x \in X$ such that $e_{i}(x)$ and $e_{j}(x)$ are both defined, we require that $\Delta_{i}\delta_{j}(x) = \Delta_{j}\delta_{i}(x) = -1$ implies $e_{i}e_{j}^{2}e_{i}(x) = e_{j}e_{i}^{2}e_{j}(x)$ and $\nabla_{i}\varepsilon_{j}(y) = \nabla_{j}\varepsilon_{i}(y) = -1$ where $y = e_{i}e_{j}^{2}e_{i}(x) = e_{j}e_{i}^{2}e_{j}(x)$.  
\end{enumerate}

Dually, we have the additional two requirements for $X$ to be \textit{A-regular},
\begin{enumerate}
\item[(S5')] $\nabla_{i}\varepsilon_{j}(x) = 0$ implies $f_{i}f_{j}(x) = f_{j}f_{i}(x)$ and $\Delta_{j}\delta_{i}(y) = 0$ where $y = f_{i}f_{j}(x) = f_{j}f_{i}(x)$.
\item [(S6')] $\nabla_{i}\varepsilon_{j}(x) = \nabla_{j}\varepsilon_{i}(x) = -1$ implies $f_{i}f_{j}^{2}f_{i}(x) = f_{j}f_{i}^{2}f_{j}(x)$ and $\Delta_{i}\delta_{j}(y) = \Delta_{j}\delta_{i}(y)$ where $y = f_{i}f_{j}^{2}f_{i}(x) = f_{j}f_{i}^{2}f_{j}(x)$.
\end{enumerate}

\end{definition}
In \cite{Stembridge}, Stembridge proves the following:
\begin{thm}[\cite{Stembridge}]
If $A$ is a symmetrizable Cartan matrix, then the crystal graph of the irreducible highest weight $U_{q}(A)$ module of highest weight $\lambda$, $V(\lambda)$, is A-regular.
\end{thm}

All crystal graphs in this paper are that of an irreducible highest weight representation and therefore are $A$-regular so the Stembridge axioms hold.

\begin{definition}
If we have an $x$ such that (S5') holds, then we say there is a \textit{degree two Stembridge relation} upward from $x$, namely $f_{i}f_{j}(x) = f_{j}f_{i}(x)$. Similarly, if we have an $x$ such that (S6') holds, we say there is a \textit{degree four Stembridge relation} upward from $x$, namely $f_{i}f_{j}^{2}f_{i}(x) = f_{j}f_{i}^{2}f_{j}(x)$. Dually, if we have an $x$ such that (S5) holds, then we say there is a degree two Stembridge relation downward from $x$. Similarly, if we have an $x$ such that (S6) holds, we say there is a degree four Stembridge relation downward from $x$. 
\end{definition}
Below in Figure \ref{stembridgerelations} are visualizations of the degree two and degree four Stembridge relations. 
\begin{figure}[h]
\centering
\begin{tikzpicture}
[auto,
vertex/.style={circle,draw=black!100,fill=black!,thick,inner sep=1pt,minimum size=.01mm}]
    \node (0) at (0,-1) {(a)};
    \node (1) at (0,0)  [vertex, label=below:\tiny{$x$}] {};
    \node (2) at (-.5,.75) [vertex] {};
    \node (3) at (.5,.75) [vertex] {};
    \node (4) at (0,1.5) [vertex] {};
    
    \path [->] (1) edge node[left] {\tiny{$i$}} (2);
    \path [->](1) edge node[right] {\tiny{$j$}} (3);
    \path [->](2) edge node[left] {\tiny{$j$}} (4);
    \path [->](3) edge node[right] {\tiny{$i$}} (4);
    
\end{tikzpicture}
\qquad
\begin{tikzpicture}
[auto,
vertex/.style={circle,draw=black!100,fill=black!,thick,inner sep=0pt,minimum size=1mm}]
    \node (0) at (0,-1) {(b)};    
    \node (1) at (0,0) [vertex, label=below:\tiny{$x$}] {};
    \node (2) at (-.5,.75) [vertex] {};
    \node (3) at (.5,.75) [vertex] {};
    \node (4) at (-.5,1.5) [vertex] {};
    \node (5) at (.5,1.5) [vertex] {};
    \node (6) at (-.5,2.25) [vertex] {};
    \node (7) at (.5,2.25) [vertex] {};
    \node (8) at (0,3) [vertex] {};

    \path [->] (1) edge node[left] {\tiny{$i$}} (2);
    \path [->](1) edge node[right] {\tiny{$j$}} (3);
    \path [->](2) edge node[left] {\tiny{$j$}} (4);
    \path [->](3) edge node[right] {\tiny{$i$}} (5);
    \path [->](4) edge node[left] {\tiny{$j$}} (6);
    \path [->](5) edge node[right] {\tiny{$i$}} (7);
    \path [->](6) edge node[left] {\tiny{$i$}} (8);
    \path [->](7) edge node[right] {\tiny{$j$}} (8);
    
\end{tikzpicture}
\caption{(a) The degree two Stembridge relation, and (b) the degree four Stembridge relation.} \label{stembridgerelations}
\end{figure}
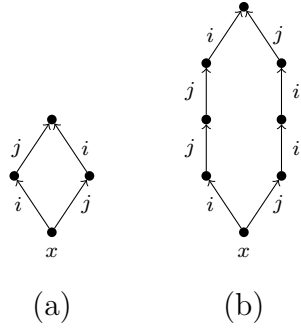

We now consider the doubly laced case. 
In \cite{Sternberg}, Sternberg proves a conjecture of Stembridge providing a description of the local structure of crystals arising from doubly laced Kac Moody algebras. 

\begin{thm}[\cite{Sternberg}]
Let $\mathfrak{g}$ be a doubly laced algebra, let $\mathcal{B}$ be the crystal graph of an irreducible highest weight module of $\mathfrak{g}$, and let $x$ be a vertex of $\mathcal{B}$ such that $f_{i}(x)$ is defined and $f_{j}(x)$ is defined where $f_{i}$ and $f_{j}$ are two distinct crystal operators. Then exactly one of the following is true:
\begin{enumerate}
\item $f_{i}f_{j}(x) = f_{j}f_{i}(x)$,
\item $f_{i}f_{j}^{2}f_{i}(x) = f_{j}f_{i}^{2}f_{j}(x)$,
\item $f_{i}f_{j}^{3}f_{i}(x) = f_{j}f_{i}f_{j}f_{i}f_{j}(x) = f_{j}^{2}f_{i}^{2}f_{j}(x)$,
\item $f_{i}f_{j}^{3}f_{i}^{2}f_{j}(x) = f_{i}f_{j}^{2}f_{i}f_{j}f_{i}f_{j}(x) = f_{j}f_{i}^{2}f_{j}^{3}f_{i}(x) = f_{j}f_{i}f_{j}f_{i}f_{j}^{2}f_{i}(x)$.
\end{enumerate}
The equivalent statement with the crystal operators $e_{i}$ and $e_{j}$ also holds.
\end{thm}

\begin{definition}
If we have $x \in \mathcal{B}$ such that \[f_{i}f_{j}^{3}f_{i}(x) = f_{j}f_{i}f_{j}f_{i}f_{j}(x) = f_{j}^{2}f_{i}^{2}f_{j}(x),\] then we say there is a \textit{degree five Sternberg relation} upward from $x$. Similarly, if we have $x \in \mathcal{B}$ such that \[f_{i}f_{j}^{3}f_{i}^{2}f_{j}(x) = f_{i}f_{j}^{2}f_{i}f_{j}f_{i}f_{j}(x) = f_{j}f_{i}^{2}f_{j}^{3}f_{i}(x) = f_{j}f_{i}f_{j}f_{i}f_{j}^{2}f_{i}(x),\] then we say that there is a \textit{degree seven Sternberg relation} upward from $x$. Dually, when these relations occur involving the $e_{i}$'s, we say we have a degree five or degree seven Sternberg relation downward from $x$. 
\end{definition}
See Figure \ref{sternbergrelations} for visualizations of the degree five and degree seven Sternberg relations.

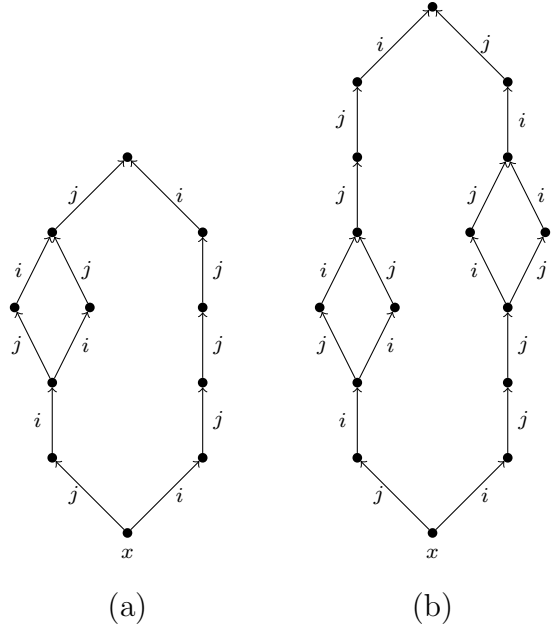
\begin{figure}[h]
\centering
	\begin{tikzpicture}
[auto,
vertex/.style={circle,draw=black!100,fill=black!,thick,inner sep=0pt,minimum size=1mm}]
    \node (0) at (0,-1) {(a)}; 
    \node (a) at (0,0) [vertex, label=below:\tiny{$x$}] {};
    \node (b) at (1,1) [vertex] {};
    \node (c) at (1,2) [vertex] {};
    \node (d) at (1,3) [vertex] {};
    \node (e) at (1,4) [vertex] {};
    \node (f) at (0,5) [vertex] {};
    \node (g) at (-1,1) [vertex] {};
    \node (h) at (-1,2) [vertex] {};    
    \node (i) at (-.5,3) [vertex] {};
    \node (j) at (-1.5,3) [vertex] {};
    \node (k) at (-1,4) [vertex] {};

    \path [->] (a) edge node[right] {\tiny{$i$}} (b);
    \path [->](b) edge node[right] {\tiny{$j$}} (c);
    \path [->](c) edge node[right] {\tiny{$j$}} (d);
    \path [->](d) edge node[right] {\tiny{$j$}} (e);
    \path [->](e) edge node[right] {\tiny{$i$}} (f);
    \path [->](a) edge node[left] {\tiny{$j$}} (g);
    \path [->](g) edge node[left] {\tiny{$i$}} (h);
    \path [->](h) edge node[right] {\tiny{$i$}} (i); 
    \path [->](h) edge node[left] {\tiny{$j$}} (j); 
    \path [->](i) edge node[right] {\tiny{$j$}} (k); 
    \path [->](j) edge node[left] {\tiny{$i$}} (k); 
    \path [->](k) edge node[left] {\tiny{$j$}} (f); 
\end{tikzpicture}
\qquad
\begin{tikzpicture}
[auto,
vertex/.style={circle,draw=black!100,fill=black!,thick,inner sep=0pt,minimum size=1mm}]
    \node (0) at (0,-1) {(b)}; 
    \node (a) at (0,0) [vertex, label=below:\tiny{$x$}] {};
    \node (b) at (1,1) [vertex] {};
    \node (c) at (1,2) [vertex] {};
    \node (d) at (1,3) [vertex] {};
    \node (e) at (.5,4) [vertex] {};
    \node (f) at (1.5,4) [vertex] {};
    \node (g) at (1,5) [vertex] {};
    \node (h) at (1,6) [vertex] {};    
    \node (i) at (0,7) [vertex] {};
    \node (k) at (-1,1) [vertex] {};
    \node (l) at (-1,2) [vertex] {};
    \node (m) at (-.5,3) [vertex] {};    
    \node (n) at (-1.5,3) [vertex] {};
    \node (o) at (-1,4) [vertex] {};
    \node (p) at (-1,5) [vertex] {};
    \node (q) at (-1,6) [vertex] {};
    
    \path [->] (a) edge node[right] {\tiny{$i$}} (b);
    \path [->](b) edge node[right] {\tiny{$j$}} (c);
    \path [->](c) edge node[right] {\tiny{$j$}} (d);
    \path [->](d) edge node[right] {\tiny{$j$}} (f);
    \path [->](d) edge node[left] {\tiny{$i$}} (e);
    \path [->](e) edge node[left] {\tiny{$j$}} (g);
    \path [->](f) edge node[right] {\tiny{$i$}} (g);
    \path [->](g) edge node[right] {\tiny{$i$}} (h); 
    \path [->](h) edge node[right] {\tiny{$j$}} (i); 
    \path [->](a) edge node[left] {\tiny{$j$}} (k); 
    \path [->](k) edge node[left] {\tiny{$i$}} (l); 
    \path [->](l) edge node[right] {\tiny{$i$}} (m); 
    \path [->](l) edge node[left] {\tiny{$j$}} (n); 
    \path [->](m) edge node[right] {\tiny{$j$}} (o); 
    \path [->](n) edge node[left] {\tiny{$i$}} (o); 
    \path [->](o) edge node[left] {\tiny{$j$}} (p);
    \path [->](p) edge node[left] {\tiny{$j$}} (q);
    \path [->](q) edge node[left] {\tiny{$i$}} (i);
\end{tikzpicture}
\caption{(a) The degree five Sternberg relation, and (b) the degree seven Sternberg relation.} \label{sternbergrelations}
\end{figure}

\subsection{Basics of partially ordered sets (posets)} 
A \textit{partially ordered set} $P$ (or poset) is a set together with a partial order $\leq$ such that the partial order satisfies the following:
\begin{enumerate}
\item reflexivity: for all $s \in P, s \leq s$.
\item antisymmetry: for $s, t \in P$, if $s \leq t$ and $t \leq s$, then $s = t$.
\item transitivity: for $s, t, u \in P$, if $s \leq t$ and $t \leq u$, then $s \leq u$. 
\end{enumerate}
Given a subset $Q\subseteq P$, we say that $Q$ is an \textit{induced subposet} of $P$ if for $s, t \in Q$, we have $s \leq t$ in $Q$ if and only if $s \leq t$ in $P$ as well. In this paper, when we say subposet we mean an induced subposet. Often, a poset is generated by \textit{cover relations} or \textit{covers}, denoted by $u \lessdot v$ or $u \prec v$, where we say that $v$ covers $u$ if $ u < v$ and there is no element $w \in P$ such that $u < w < v$. An (open) \textit{interval} is a subposet of $P$ of the form $(u,v) = \{s \in P : u < s < v \}$ defined whenever $u < v$, whereas a closed interval is defined to be $[u,v] = \{s \in P : u \leq s \leq v \}$. We say that $P$ has a \textit{minimum element}, usually denoted $\hat{0}$, if there exists some element $ \hat{0} \in P$ such that $\hat{0} \leq u$ for all $u \in P$. Similarly, $P$ has a \textit{maximum element}, usually denoted $\hat{1}$ if there exists a $\hat{1} \in P$ such that $ u \leq \hat{1}$ for all $u \in P$. A \textit{chain} is a poset in which any two elements are comparable. A subset $C$ of $P$ is a chain if it is a chain when considered as a subposet of $P$. A \textit{saturated chain} from $u$ to $v$ is a series of cover relations $u=u_{0} \lessdot u_{1} \lessdot \dotsb \lessdot u_{k} = v$. We say that a finite poset is \textit{graded} if for all $u \leq v$, every saturated chain from $u$ to $v$ has the same number of cover relations, and we call this number the \textit{rank} of the interval $[u,v]$. The \textit{rank of an element} $x \in P$ is the rank of the interval $[\hat{0}, x]$. The \textit{Hasse diagram} of a finite poset $P$ is the graph whose vertices are elements of $P$ with an edge between $x$ and $y$ whenever $x \lessdot y$.
 

Recall that the \textit{M\"obius function}, $\mu_P$ of a poset $P$ may be defined recursively as follows:\ $\mu(u,u) = 1$, for all $u \in P$, and $\mu(u,v) = - \sum_{u \leq t < v} \mu(u,t)$, for all $u < v \in P$. Given any poset $P$, we can construct the \textit{order complex} $\Delta (P)$, which  is the abstract simplicial complex whose $i$-dimensional faces are the chains $x_{0} < x_{1} < \dotsb < x_{i}$ of comparable elements in $P$. Let $\Delta_{P}(u,v)$ denote the order complex of the subposet consisting of the open interval $(u,v)$. One reason to be interested in the order complex of a poset is the connection between the M\"obius function of a poset $P$ and the Euler characteristic of the order complex $\Delta(P)$, discussed e.g. in \cite{Rota}.
\begin{thm} \label{mobiuseuler}
$\mu_{P}(\hat{0},\hat{1}) = \tilde{\chi}(\Delta(P))$.
\end{thm}

In this paper, the posets we study come from crystals. More specifically, we study the crystal graphs of crystals of highest weight representations of finite classical Cartan type. We can view these crystal graphs as posets with covering relations given by $ u \lessdot v$ whenever $ v = f_{i}(u)$ for some $i \in I$. This extends to a partial order on the crystal graph where $u < v$ whenever there is a directed path from $u$ to $v$. We color the edge of the covering relation given by $f_{i}(u) = v$ with the color $i$ giving us the structure of an edge colored poset. We call these posets \textit{crystal posets}. The crystal graph is the Hasse diagram of the crystal poset. In this way, we can talk about the combinatorics of the crystal poset. Namely, we can refer to intervals within a given crystal and compute the M\"obius function of such an interval. For more background information on posets, see \cite{StanleyEC1}.

\subsection{Discrete Morse functions}
Discrete Morse theory was introduced in \cite{Forman} by Forman as a tool to study the homotopy type and homology groups of (primarily finite) CW complexes. 

In 2000, Chari gave a combinatorial reformulation in the case of regular CW complexes in terms of acyclic matchings on their face posets in \cite{Chari}, which is what we will use in this paper. A matching on the Hasse diagram of a face poset is \textit{acyclic} if the directed graph obtained by directing matching edges upward and all other edges downward has no directed cycles. It is shown, for example in \cite{Hersh}, that whenever a face poset has an acyclic matching, there is a nonempty set of associated discrete Morse functions on the corresponding complex.

In this paper, we will apply discrete Morse theory to simplicial complexes associated to crystal posets. Let $\Delta$ be a simplicial complex. We denote a $d$-simplex $\alpha$ by $\alpha^{(d)}$. 
\begin{definition}
A \textit{discrete Morse function} on a simplicial complex $\Delta$ is a function $f: \Delta \rightarrow \rr$ such that for each $d$-dimensional simplex, $\alpha^{(d)} \in \Delta$,
\begin{enumerate}
\item $| \{\beta^{(d+1)} \supseteq \alpha | f(\beta) \leq f(\alpha) \} | \leq 1,$
\item $| \{ \gamma^{(d-1)} \subseteq \alpha | f(\gamma) \geq f(\alpha) \}| \leq 1.$
\end{enumerate}
\end{definition}

\begin{definition}
A simplex $\alpha$ is called a \textit{critical cell} if $| \{\beta^{(d+1)} \supseteq \alpha | f(\beta) \leq f(\alpha) \} | = 0$ and $| \{ \gamma^{(d-1)} \subseteq \alpha | f(\gamma) \geq f(\alpha) \}| = 0$. Equivalently, a simplex $\alpha$ is called a critical cell if it is left unmatched by the matching on the face poset.
\end{definition}

One of the reasons discrete Morse functions are useful is the following theorem.

\begin{thm}[\cite{Forman}]\label{discrete morse}
Suppose $\Delta$ is a simplicial complex with a discrete Morse function. Then $\Delta$ is homotopy equivalent to a CW-complex with exactly one cell of dimension $d$ for each critical cell of dimension $d$ with respect to this choice of discrete Morse function. 
\end{thm}

We deviate slightly from Forman's conventions in a way that is typical in combinatorics. We allow the empty set to be in the domain of our discrete Morse function $f$, as well as in the face posets on which we construct acyclic matchings. By doing so, we must express our results in terms of reduced Euler characteristic and reduced homology. When using this reduced version of discrete Morse theory, we denote the number of critical cells of dimension $i$ by $\tilde{m}_{i}$ and the reduced Betti numbers by $\tilde{b}_{i}$.

\begin{rmk}
From Theorem \ref{discrete morse}, rephrased to use reduced Betti numbers and Morse numbers, we can immediately deduce that if a discrete Morse function has exactly one critical cell of dimension $i$ and no other critical cells, then our original simplicial complex is homotopy equivalent to an $i$-dimensional sphere.
\end{rmk}

In \cite{HershBabson}, Babson and Hersh introduced lexicographic discrete Morse functions as a tool to study the topology of order complexes of partially ordered sets with $\hat{0}$ and $\hat{1}$. This is what we will use to study crystal posets. They show how each facet contributes at most one critical cell, describing these using the minimal skipped intervals discussed shortly. 

Before we describe how to construct lexicographic discrete Morse functions, we explain some of the useful properties they will have. By virtue of properties of lexicographic orders on saturated chains, the lexicographic discrete Morse functions will have relatively few critical cells. If the attachment of the facet corresponding to some saturated chain does not change the homotopy of the subcomplex of our order complex built so far, then this step does not introduce any critical cells.

Now we review lexicographic discrete Morse functions in general. This will rely on a notion of rank within a chain that does not require the poset to be graded. However, in this paper, the crystal posets we are interested in are graded by the weight function, as seen in Lemma \ref{saturated chain same labels}, simplifying the grading in a chain. 


Given a poset $P$ graded of rank $n$, let $\beta$ be an integer labeling on the edges of the Hasse diagram of $P$ such that $\beta(u,v) \neq \beta(u,w)$ whenever $v \neq w$. Each facet of $\Delta(P)$ corresponds to a saturated chain, $\hat{0} \lessdot u_{1} \lessdot \cdots \lessdot u_{k} \lessdot \hat{1}$ in $P$. For each saturated chain we read off the label sequence $(\beta(\hat{0},u_{1}), \beta(u_{1},u_{2}), \cdots , \beta(u_{k},\hat{1}))$ and order these lexicographically. This labeling gives rise to a total order on the facets $F_{1}, . . . , F_{k}$ of the order complex. 
By virtue of the fact that we attach facets in a lexicographic order, each maximal face in $\overline{F}_{j} \cap (\cup_{i<j} \overline{F}_{i})$ has rank set of the form $1,. . . , i, j, ..., n$ for $j > i+1$ i.e. it omits a single interval of consecutive ranks. We call this rank interval $[i+1,j-1]$ a \textit{minimal skipped interval} of $F_{j}$ with support $i+1, ... , j-1$ and height $j-i-1$. For a given facet $F_{j}$, we call the collection of minimal skipped intervals the \textit{interval system of $F_{j}$}. 

\begin{rmk}\label{general MSI description}
In order to determine the minimal skipped intervals for a given saturated chain $M$ corresponding to some facet $F_{j}$, we consider each cover relation $u \lessdot v$ as we travel up $M$. At each cover relation $u \lessdot v$, we check if there is a lexicographically earlier cover relation $u \lessdot v'$ upward from $u$. If so, we obtain a maximal face in $\overline{F}_{j} \cap(\cup_{i<j} \overline{F}_{i})$, and hence a minimal skipped interval, by taking the intersection of $\overline{F}_{j}$ with the closure of any facet $F_{i'}$ that includes $u \lessdot v'$, that agrees with $F_{j}$ below $u$ and agrees with $F_{j}$ above $w \in F_{j}$ for some $w > v'$ of minimal rank. 
\end{rmk}

When our poset has some natural labeling, like that of our crystal posets, it is often possible to classify its minimal skipped intervals. 

Any face in $\overline{F}_{j} \setminus (\cup_{i<j} \overline{F}_{i})$ must include at least one rank from each of the minimal skipped intervals of $F_{j}$. For each $j$, an acyclic matching on the set of faces in $\overline{F}_{j} \setminus (\cup_{i<j} \overline{F}_{i})$ is constructed in \cite{HershBabson} in terms of the interval system. The union of these matchings is acyclic on the entire Hasse diagram of the face poset of the order complex of $P$. For more about this acyclic matching, see \cite{Hersh}.

\begin{rmk}\label{proper part}
In actuality, we study the order complexes of the proper parts of our posets; if $P$ has a $\hat{0}$ and $\hat{1}$ then $\Delta(P)$ is contractible as it is a cone. We use the $\hat{0}$ and $\hat{1}$ in the lexicographic discrete Morse functions in a bookkeeping role. More specifically, $\hat{0}$ and $\hat{1}$ are needed to record the labels of covering relations upward from $\hat{0}$ and upward towards $\hat{1}$. In particular, when we refer to fully covered saturated chains, the ranks of $\hat{0}$ and $\hat{1}$ are not covered.
\end{rmk}

A facet $F_{j}$ will contribute a critical cell if and only if the interval system of $F_{j}$ covers all ranks in $F_{j}$, keeping Remark \ref{proper part} in mind, after the truncation algorithm described below. In this case we say that the corresponding saturated chain is \textit{fully covered}. The dimension of such a critical cell is one less than the number of minimal skipped intervals in the interval system after the truncation algorithm. This truncation algorithm is needed when the interval system of some facet $F_{j}$ covers all ranks but there are overlapping minimal skipped intervals.  Otherwise the truncated system equals the original system. 

For the truncation algorithm, we begin with our interval system, $I$, and initialize the truncated system, which we call $J$, to be the empty set. Then, we repeatedly move the minimum interval in $I$ to the truncated system $J$ and truncate all other elements of $I$ to eliminate any overlap with the minimum interval in $I$ being moved to $J$ at this step. Here, by minimum we mean the minimal skipped interval containing the element of smallest rank. Next, remove any intervals in $I$ that are no longer minimal. We repeat this until there are no longer any minimal skipped intervals in $I$. We call the truncated, minimal intervals obtained by this algorithm the \textit{$J$-intervals of $F_{j}$}. By construction, these are non-overlapping. If the $J$-intervals cover all ranks of $F_{j}$, then $F_{j}$ contributes a critical cell. We get this critical cell by taking the lowest rank element of each of the $J$-intervals. Otherwise $F_{j}$ does not contribute any critical cells. For a more detailed background on lexicographic discrete Morse functions, see \cite{Hersh}.

\section{Consequences of the Stembridge axioms}
In this section, we deduce consequences of the Stembridge axioms regarding relations among crystal operators in both the simply laced and doubly laced cases. These axioms give restrictions on which Stembridge/Sternberg relations can occur among two given crystal operators for crystals coming from highest weight representations in the simply and doubly laced cases. In addition, we prove that crystals of types $B_{2}$ and $C_{2}$ are not lattices due to the asymmetrical nature of the degree five Sternberg relations.

Throughout this section, we will let $A = [a_{ij}]_{i,j \in I}$ be the Cartan matrix of a symmetrizable Kac Moody algebra $\mathfrak{g}$, where $I$ is some finite index set. We first will work with simply laced Kac Moody algebras of types $A_{n}$ and $D_{n}$. The Cartan matrices for these types can be seen in Figures \ref{fig:An} and \ref{fig:Dn}.
\begin{figure}[h]
\centering
 $A = \begin{bmatrix}
    2 & -1 & 0 & 0 & \dots  & 0 & 0 \\
    -1 & 2 & -1 & 0 & \dots  & 0 & 0 \\
     0 & -1 & 2 & -1 & \dots & 0 & 0 \\
    \vdots & \vdots & \vdots & \vdots & \ddots & \vdots & \vdots \\
    0 & 0 & 0 & 0 & \dots & 2 & -1 \\
    0 & 0 & 0 & 0 & \dots  &  -1 & 2 \\
\end{bmatrix}$
  \caption{Cartan matrix of type $A_{n}$}
  \label{fig:An}
\end{figure}
\begin{figure}[h]
\centering
 $ A = \begin{bmatrix}
    2 & -1 & 0 &  \dots  & 0 & 0 & 0 & 0\\
    -1 & 2 & -1 &  \dots  & 0 & 0 & 0 & 0 \\
     0 & -1 & 2 &  \dots  & 0 & 0 & 0 & 0 \\
    \vdots & \vdots & \vdots & \ddots & \vdots & \vdots & \vdots & \vdots \\
    0 & 0 & 0 &  \dots & 2 & -1 & 0 & 0 \\
    0 & 0 & 0 &  \dots & -1 & 2 & -1 & -1 \\
    0 & 0 & 0  & \dots  & 0 & -1 & 2 & 0 \\
    0 & 0 & 0 & \dots  & 0 & -1 & 0 & 2 \\
\end{bmatrix}$
  \caption{Cartan matrix of type $D_{n}$}
  \label{fig:Dn}
\label{fig:Cartan matrices}
\end{figure}

For the Cartan matrix of type $A_{n}$, we note that $a_{i,i+1} = a_{i+1,i} = -1$ and all other off diagonal entries are zero. In particular, $a_{ij} = 0$ for all $i,j \in [n]$ such that $|i-j| > 1$. Therefore, using axioms (S3) and (S4), we must have that  for any vertex $x$ where $e_{i}(x) \neq 0$ and $|i-j| > 1$ in a crystal graph of type $A_{n}$, $\Delta_{i}\delta_{j}(x) = \Delta_{i}\varepsilon_{j}(x) = 0$. Similarly, we have $\Delta_{j}\delta_{i}(x) = \Delta_{j}\varepsilon_{i}(x) = 0$. As a result, when we have an $x$ with $f_{i}(x) \neq 0$ and $f_{j}(x) \neq 0$, we have the following statement regarding degree two and degree four Stembridge relations.

\begin{proposition}
Let $\mathcal{B}$ by a crystal of type $A_{n}$. Let $x \in \mathcal{B}$ such that $f_{i}(x) \neq 0$ and $f_{j}(x) \neq 0$. Then we have:
\begin{enumerate}
\item If $|i-j| >1$, then $f_{i}f_{j}(x) = f_{j}f_{i}(x)$.
\item If $|i-j| = 1$, then either $f_{i}f_{j}(x) = f_{j}f_{i}(x)$ or $f_{i}f_{j}^{2}f_{i}(x) = f_{j}f_{i}^{2}f_{j}(x)$.
\end{enumerate}
The statement remains true if we replace $f_{i}$ and $f_{j}$ with the crystal operators $e_{i}$ and $e_{j}$.
\end{proposition}

Now we consider the Cartan matrix for type $D_{n}$ from Figure \ref{fig:Dn}. Here we have that $a_{i,i+1} = a_{i+1,i} = -1$ for all $1 \leq i \leq n-2$. In addition, we also have $a_{n-2,n} = a_{n,n-2} = -1$. All other off diagonal entries are equal to zero. In particular, $a_{n-1,n} = a_{n,n-1} = 0$, which differs from the Cartan matrix of type $A_{n}$. Hence, using axioms (S3) and (S4) again, we get the following result. 

\begin{proposition}
Let $\mathcal{B}$ be a crystal of type $D_{n}$. Let $x \in \mathcal{B}$ such that $f_{i}(x) \neq 0$ and $f_{j}(x) \neq 0$. Then we have:
\begin{enumerate}
\item If $|i-j| > 1$ and $\{i,j\} \neq \{n-2,n\}$, then $f_{i}f_{j}(x) = f_{j}f_{i}(x)$.
\item If $\{i,j\} = \{n-1,n\}$, then $f_{i}f_{j}(x) = f_{j}f_{i}(x)$.
\item If $|i-j| = 1$ and $\{i,j\} \neq \{n-1,n\}$, then either $f_{i}f_{j}(x) = f_{j}f_{i}(x)$ or $f_{i}f_{j}^{2}f_{i}(x) = f_{j}f_{i}^{2}f_{j}(x)$.
\item If $\{i,j\} = \{n-2,n\}$, then either $f_{i}f_{j}(x) = f_{j}f_{i}(x)$ or $f_{i}f_{j}^{2}f_{i}(x) = f_{j}f_{i}^{2}f_{j}(x)$.
\end{enumerate}
The statement remains true if we replace $f_{i}$ and $f_{j}$ with the crystal operators $e_{i}$ and $e_{j}$.
\end{proposition}

We now move on to consider the doubly laced case. 
We recall the Cartan matrices of type $B_{n}$ and type $C_{n}$ in Figures \ref{fig:sub1} and \ref{fig:sub2}. 
\begin{figure}[h]
\centering
 $A = \begin{bmatrix}
    2 & -1 & 0 & 0 & \dots  & 0 & 0 \\
    -1 & 2 & -1 & 0 & \dots  & 0 & 0 \\
     0 & -1 & 2 & -1 & \dots & 0 & 0 \\
    \vdots & \vdots & \vdots & \vdots & \ddots & \vdots & \vdots \\
    0 & 0 & 0 & 0 & \dots & 2 & -2 \\
    0 & 0 & 0 & 0 & \dots  &  -1 & 2 \\
\end{bmatrix}$
  \caption{Cartan matrix of type $B_{n}$}
  \label{fig:sub1}
\end{figure}
\begin{figure}[h]
  \centering
 $ A' = \begin{bmatrix}
    2 & -1 & 0 & 0 & \dots  & 0 & 0 \\
    -1 & 2 & -1 & 0 & \dots  & 0 & 0 \\
     0 & -1 & 2 & -1 & \dots & 0 & 0 \\
    \vdots & \vdots & \vdots & \vdots & \ddots & \vdots & \vdots \\
    0 & 0 & 0 & 0 & \dots & 2 & -1 \\
    0 & 0 & 0 & 0 & \dots  &  -2 & 2 \\
\end{bmatrix}$
  \caption{Cartan matrix of type $C_{n}$}
  \label{fig:sub2}
\end{figure}
\label{fig:Cartan matrices}

For the Cartan matrix $A$ of type $B_{n}$, note that $a_{n-1,n} = -2$ and for the Cartan matrix $A$ of type $C_{n}$, we have $a_{n,n-1} = -2$. All the remaining superdiagonal entries $a_{i,i+1}$ and remaining subdiagonal entries $a_{i+1,i}$ in $A$ of type $B_{n}$ and $C_{n}$ are equal to $-1$. The remaining off diagonal entries in each Cartan matrix are all zero. Therefore, since crystal graphs of type $B_{n}$ and $C_{n}$ are $A-$regular, for $\{i,j\} \neq \{n-1,n\}$, by axioms (S3) and (S4) we have that for any given vertex $x$ there are only three possibilities for the triples $(a_{ij},\Delta_{i}\delta_{j}(x), \Delta_{i}\varepsilon_{j}(x))$, namely $(0,0,0), (-1,-1,0)$ or $(-1,0,-1)$. Hence, by axioms (S5)-(S6) and (S5')-(S6'), we have the following result. 

\begin{thm}\label{possible Sternberg} Let $\mathcal{B}$ be a crystal of type $B_{n}$ or $C_{n}$. Let $x \in \mathcal{B}$ such that $f_{i}(x) \neq 0$ and $f_{j}(x) \neq 0$. Then we have:
\begin{enumerate}
\item If $|i-j| > 1$, then $f_{i}f_{j}(x) = f_{j}f_{i}(x)$.
\item If $|i-j| = 1$ and $\{i,j\} \neq \{n-1,n\}$, then either $f_{i}f_{j}(x) = f_{j}f_{i}(x)$ or $f_{i}f_{j}^{2}f_{i}(x) = f_{j}f_{i}^{2}f_{j}(x)$.
\item If $\{i,j\} = \{n-1,n\}$, then any of the Stembridge or Sternberg relations are possible.
\end{enumerate}
\end{thm}

Therefore, for crystal graphs of doubly laced type, the degree four Stembridge relations can only occur among crystal operators $f_{i}$ and $f_{i+1}$ as in type $A_{n}$, and the degree five and degree seven Sternberg relations can only occur among the crystal operators $f_{n-1}$ and $f_{n}$.

Crystals of rank two algebras are often of particular interest. This is due to the result seen in \cite{KMN} which says that a crystal graph with a unique maximal vertex is the crystal graph of some representation if and only if it decomposes as the disjoint union of crystals of representations relative to the rank two subalgebras corresponding to each pair of edge colors. Therefore, we now consider crystals of type $B_{2}$ and $C_{2}$. In \cite{DKK}, it is shown that crystals of type $A_{2}$ are lattices. We show that this result does not carry over to the doubly laced case. 
\begin{thm}
Crystals of highest weight representations of types $B_{2}$ and $C_{2}$ are not lattices.
\end{thm}
\begin{proof}
This follows from the asymmetrical nature of the degree five Sternberg relations. Let $\mathcal{B}$ be the crystal of a highest weight representation of type $B_{2}$ or $C_{2}$. Let $x \in \mathcal{B}$ such that there is a degree five Sternberg relation upward from $x$. Then we have $y \in \mathcal{B}$ such that 
\begin{equation}
y = f_{1}f_{2}^{3}f_{1}(x) = f_{2}f_{1}f_{2}f_{1}f_{2}(x) = f_{2}^{2}f_{1}^{2}f_{2}(x),
\end{equation}
or
\begin{equation}
y = f_{2}f_{1}^{3}f_{2}(x) = f_{1}f_{2}f_{1}f_{2}f_{1}(x) = f_{1}^{2}f_{2}^{2}f_{1}(x).
\end{equation}
In either case, we have that $e_{1}(y) \neq 0$ and $e_{2}(y) \neq 0$. As a result, there must be a Stembridge or Sternberg relation downward from $y$. Hence, $e_{1}(y)$ and $e_{2}(y)$ will have two distinct, incomparable greatest lower bounds, one coming from the Stembridge or Sternberg relation downward from $y$ and the other being $x$. 

Similarly, if there exists $y \in \mathcal{B}$ such that there is a degree five Sternberg relation downward from $y$, then there will exist two vertices that have two distinct, incomparable least upper bounds. Hence, highest weight representations of types $B_{2}$ and $C_{2}$ are not lattices.
\end{proof}


\section{Connection between the M\"obius function of a poset and relations among crystal operators}
In this section, we prove that whenever there is an interval in a crystal poset coming from a highest weight representation of finite classical Cartan type where the M\"obius function is not equal to $-1, 0, $ or $1$, there must exist a relation among crystal operators within that interval that is not implied by Stembridge or Sternberg relations. Hersh and Lenart have previously proven this in \cite{HershLenart} in the case where the crystal is that of a highest weight representation of finite simply laced type. We extend this result to crystals of type $B_{n}$ and $C_{n}$ and in doing so, give a new proof for crystals of types $A_{n}$ and $D_{n}$. To do so, we first develop properties of crystal graphs we will need.

\begin{lemma}\label{saturated chain same labels}
Let $\mathcal{B}$ be the crystal graph of a crystal of type $\Phi$ given by a highest weight representation of finite classical Cartan type. Let $u, v \in \mathcal{B}$ such that $u < v$. Any saturated chain from $u$ to $v$ uses the same multiset of edge labels. Moreover, we can determine this multiset given wt$(u)$ and wt$(v)$.
\end{lemma}
\begin{proof}
Recall that if $y = f_{i}(x)$ then wt($y) = $ wt($x) - \alpha_{i}$ where $\alpha_{i}$ is the $i^{th}$ simple root of our root system $\Phi$. Since $u < v$, there exists some sequence of crystal operators $f_{i_{1}}, f_{i_{2}}, ... , f_{i_{k}}$ such that $v = f_{i_{k}}\cdots f_{i_{2}}f_{i_{1}}(u)$. Then we have, \[ \text{wt}(v) = \text{wt}(u) - \sum_{j=1}^{k} \alpha_{i_{j}}. \]
Suppose by way of contradiction there exists another distinct sequence of crystal operators $f_{l_{1}}, f_{l_{2}}, ... , f_{l_{m}}$ such that $v = f_{l_{m}}\cdots f_{l_{2}}f_{l_{1}}(u)$. Then we have \[ \text{wt}(u) - \text{wt}(v) = \sum_{j=1}^{k} \alpha_{i_{j}} = \sum_{n=1}^{m} \alpha_{l_{n}}.\] Since the set of simple roots $\{ \alpha_{i} \}_{i \in I}$ is linearly independent, we must have that $\{ \alpha_{i_{1}}, ... , \alpha_{i_{k}}\} = \{ \alpha_{l_{1}} , ... , \alpha_{l_{m}} \}$. Therefore, the same crystal operators are used with the same multiplicities along any saturated chain from $u$ to $v$. In addition, by writing the vector wt$(u) - $ wt$(v)$ in terms of the simple roots, we can see exactly how many times each crystal operator $f_{i}$ is applied along any saturated chain from $u$ to $v$. 
\end{proof}

\begin{rmk}
This tells us that crystal posets are graded since every saturated chain in a given interval $[u,v]$ will have the same length. 
\end{rmk}



We now work towards showing that for $[u,v] \subseteq \mathcal{B}$ of simply laced (respectively, doubly laced) type with the property that all relations among crystal operators are implied by Stembridge (respectively, Stembridge or Sternberg) relations, we must have that $\mu(u,v) \in \{-1,0,1\}$. We do so by constructing a lexicographic discrete Morse function on $\Delta(u,v)$ that has at most one critical cell. Recall that we have a critical cell if and only if we have a fully covered saturated chain. Therefore, to prove the result, we give a method to find the unique fully covered saturated chain in the given interval $[u,v]$ when such a chain exists.

We use the natural edge coloring on crystal posets as labels to construct the lexicographic discrete Morse function i.e. the edge from $x$ to $y$ is labeled $i$ if $y = f_{i}(x)$. Throughout this section, we assume that for any interval $[u,v]$, all relations among crystal operators are implied by Stembridge relations, in the simply laced case, and Stembridge and Sternberg relations, in the doubly laced case.

\begin{definition}
Let $\mathcal{B}$ be the crystal of a highest weight representation of finite classical Cartan type. Let $[u,v] \subseteq \mathcal{B}$. If $\mathcal{B}$ is of simply laced type and all relations among crystal operators within $[u,v]$ are implied by Stembridge relations, then we say that $[u,v]$ is a \textit{Stembridge only interval}. If $\mathcal{B}$ is of doubly laced type and all relations among crystal operators are implied by Stembridge or Sternberg relations, then we say that $[u,v]$ is a \textit{Stembridge and Sternberg only interval}.
\end{definition}

As we are assuming all intervals in this section are either Stembridge only or Stembridge and Sternberg only intervals, we have that each minimal skipped interval (as described in Remark \ref{general MSI description}) in our lexicographic discrete Morse function will arise from a Stembridge or Sternberg relation. Hence, all minimal skipped intervals will be of the forms seen in Figure \ref{Stem type A} and Figure \ref{Sternberg relations}.

In the case where $\mathcal{B}$ is the crystal of a highest weight representation of type $A_{n}$ or $D_{n}$, all minimal skipped intervals are of the form seen in Figure \ref{Stem type A}.
Assume $i < j $. Note we include the special degree four Stembridge relation that can occur among the crystal operators $f_{n-2}$ and $f_{n}$ in type $D_{n}$ in the third picture. The saturated chain in red, namely the chain $x \lessdot u_{0} \lessdot y$ in the first picture and $x \lessdot u_{0} \lessdot u_{1} \lessdot u_{2} \lessdot y$ in the second and third, represents the piece of the Stembridge relation that may be on a fully covered saturated chain as it is the lexicographically second chain. The lexicographically earlier chain with vertices labeled by $v_{i}$ will lead to a minimal skipped interval covering the single rank corresponding to the vertex $u_{0}$ in the first picture, and the ranks corresponding to the vertices $u_{0}, u_{1},$ and $u_{2}$ in the second and third pictures.
\begin{figure}[h]
\centering
\begin{tikzpicture}
[auto,
vertex/.style={circle,draw=black!100,fill=black!,thick,inner sep=1pt,minimum size=.01mm}]
    \node (1) at (0,0)  [vertex, label=below:\tiny{$x$}] {};
    \node (2) at (-.5,.75) [vertex, label=left:\tiny{$v_{0}$}] {};
    \node (3) at (.5,.75) [vertex, label=right:\tiny{$u_{0}$}] {};
    \node (4) at (0,1.5) [vertex, label=above:\tiny{$y$}] {};
    \node (0) at (0,-1) {(i)};
    
    \path [->] (1) edge node[left] {\tiny{$i$}} (2);
    \path [draw=red, ->](1) edge node[right] {\tiny{$j$}} (3);
    \path [->](2) edge node[left] {\tiny{$j$}} (4);
    \path [draw=red, ->](3) edge node[right] {\tiny{$i$}} (4);
    
\end{tikzpicture}
\qquad
\begin{tikzpicture}
[auto,
vertex/.style={circle,draw=black!100,fill=black!,thick,inner sep=0pt,minimum size=1mm}]
    \node (1) at (0,0) [vertex, label=below:\tiny{$x$}] {};
    \node (2) at (-.5,.75) [vertex, label=left:\tiny{$v_{0}$}] {};
    \node (3) at (.5,.75) [vertex, label=right:\tiny{$u_{0}$}] {};
    \node (4) at (-.5,1.5) [vertex, label=left:\tiny{$v_{1}$}] {};
    \node (5) at (.5,1.5) [vertex, label=right:\tiny{$u_{1}$}] {};
    \node (6) at (-.5,2.25) [vertex, label=left:\tiny{$v_{2}$}] {};
    \node (7) at (.5,2.25) [vertex, label=right:\tiny{$u_{2}$}] {};
    \node (8) at (0,3) [vertex, label=above:\tiny{$y$}] {};
    \node (0) at (0,-1) {(ii)};

    \path [->] (1) edge node[left] {\tiny{$i$}} (2);
    \path [draw=red, ->](1) edge node[right] {\tiny{$i+1$}} (3);
    \path [->](2) edge node[left] {\tiny{$i+1$}} (4);
    \path [draw=red, ->](3) edge node[right] {\tiny{$i$}} (5);
    \path [->](4) edge node[left] {\tiny{$i+1$}} (6);
    \path [draw=red, ->](5) edge node[right] {\tiny{$i$}} (7);
    \path [->](6) edge node[left] {\tiny{$i$}} (8);
    \path [draw=red, ->](7) edge node[right] {\tiny{$i+1$}} (8);
    
\end{tikzpicture}
\qquad
\begin{tikzpicture}
[auto,
vertex/.style={circle,draw=black!100,fill=black!,thick,inner sep=0pt,minimum size=1mm}]
    \node (1) at (0,0) [vertex, label=below:\tiny{$x$}] {};
    \node (2) at (-.5,.75) [vertex, label=left:\tiny{$v_{0}$}] {};
    \node (3) at (.5,.75) [vertex, label=right:\tiny{$u_{0}$}] {};
    \node (4) at (-.5,1.5) [vertex, label=left:\tiny{$v_{1}$}] {};
    \node (5) at (.5,1.5) [vertex, label=right:\tiny{$u_{1}$}] {};
    \node (6) at (-.5,2.25) [vertex, label=left:\tiny{$v_{2}$}] {};
    \node (7) at (.5,2.25) [vertex, label=right:\tiny{$u_{2}$}] {};
    \node (8) at (0,3) [vertex, label=above:\tiny{$y$}] {};
     \node (0) at (0,-1) {(iii)};

    \path [->] (1) edge node[left] {\tiny{$n-2$}} (2);
    \path [draw=red, ->](1) edge node[right] {\tiny{$n$}} (3);
    \path [->](2) edge node[left] {\tiny{$n$}} (4);
    \path [draw=red, ->](3) edge node[right] {\tiny{$n-2$}} (5);
    \path [->](4) edge node[left] {\tiny{$n$}} (6);
    \path [draw=red, ->](5) edge node[right] {\tiny{$n-2$}} (7);
    \path [->](6) edge node[left] {\tiny{$n-2$}} (8);
    \path [draw=red, ->](7) edge node[right] {\tiny{$n$}} (8);
    
\end{tikzpicture}
\caption{Structure of minimal skipped intervals in simply laced case}
\label{Stem type A}
\end{figure}
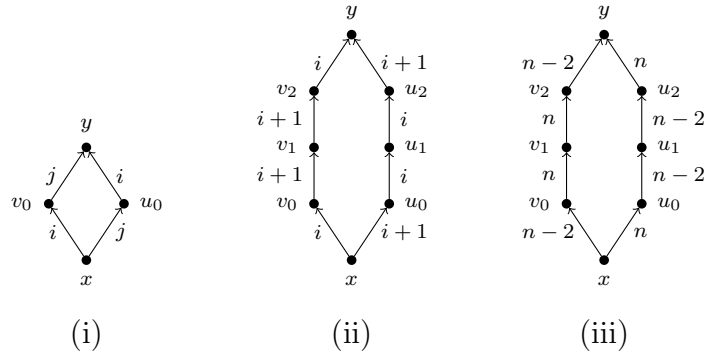
When we have a minimal skipped interval that arises from a Stembridge relation as in Figure \ref{Stem type A}, we say the minimal skipped interval involves the crystal operators, e.g. the minimal skipped interval in (ii) involves the crystal operators $f_{i}$ and $f_{i+1}$.

When $\mathcal{B}$ is the crystal of a highest weight representation of finite doubly laced type, in addition to the Stembridge relations, minimal skipped intervals may also arise from the degree five and degree seven Sternberg relations. By Theorem \ref{possible Sternberg}, we know the degree five and degree seven Sternberg relations can only occur upward from some vertex $x$ where $f_{n-1}(x) \neq 0$ and $f_{n}(x) \neq 0$. Therefore, we say that minimal skipped intervals arising from Sternberg relations involve the crystal operators $f_{n-1}$ and $f_{n}$. The possible Sternberg relations are shown below where the saturated chains in red, namely those with vertices labeled by $u_{i}$, represent the piece of the Sternberg relation that may be on a fully covered saturated chain as described above in the simply laced case. 
\begin{figure}[h]
\centering
	\begin{tikzpicture}
[auto,
vertex/.style={circle,draw=black!100,fill=black!,thick,inner sep=0pt,minimum size=1mm}]
    \node (a) at (0,0) [vertex, label=below:\tiny{$x$}] {};
    \node (b) at (1,1) [vertex, label=right:\tiny{$u_{0}$}] {};
    \node (c) at (1,2) [vertex, label=right:\tiny{$u_{1}$}] {};
    \node (d) at (1,3) [vertex, label=right:\tiny{$u_{2}$}] {};
    \node (e) at (1,4) [vertex, label=right:\tiny{$u_{3}$}] {};
    \node (f) at (0,5) [vertex, label=above:\tiny{$y$}] {};
    \node (g) at (-1,1) [vertex] {};
    \node (h) at (-1,2) [vertex] {};    
    \node (i) at (-.5,3) [vertex] {};
    \node (j) at (-1.5,3) [vertex] {};
    \node (k) at (-1,4) [vertex] {};
    \node (5) at (3, 2.5) {or};
    \node (1) at (0,-1) {(i)};

    \path [draw=red, ->] (a) edge node[right] {\tiny{$n$}} (b);
    \path [draw=red, ->](b) edge node[right] {\tiny{$n-1$}} (c);
    \path [draw=red, ->](c) edge node[right] {\tiny{$n-1$}} (d);
    \path [draw=red, ->](d) edge node[right] {\tiny{$n-1$}} (e);
    \path [draw=red, ->](e) edge node[right] {\tiny{$n$}} (f);
    \path [->](a) edge node[left] {\tiny{$n-1$}} (g);
    \path [->](g) edge node[left] {\tiny{$n$}} (h);
    \path [->](h) edge node[right] {\tiny{$n$}} (i); 
    \path [->](h) edge node[left] {\tiny{$n-1$}} (j); 
    \path [->](i) edge node[right] {\tiny{$n-1$}} (k); 
    \path [->](j) edge node[left] {\tiny{$n$}} (k); 
    \path [->](k) edge node[left] {\tiny{$n-1$}} (f); 
\end{tikzpicture}
\qquad
\begin{tikzpicture}
[auto,
vertex/.style={circle,draw=black!100,fill=black!,thick,inner sep=0pt,minimum size=1mm}]
    \node (a) at (0,0) [vertex, label=below:\tiny{$x$}] {};
    \node (b) at (1,.7) [vertex, label=right:\tiny{$u_{0}$}] {};
    \node (c) at (1,1.4) [vertex, label=right:\tiny{$u_{1}$}] {};
    \node (d) at (1,2.1) [vertex, label=right:\tiny{$u_{2}$}] {};
    \node (e) at (.5,2.8) [vertex] {};
    \node (f) at (1.5,2.8) [vertex, label=right:\tiny{$u_{3}$}] {};
    \node (g) at (1,3.5) [vertex, label=right:\tiny{$u_{4}$}] {};
    \node (h) at (1,4.2) [vertex, label=right:\tiny{$u_{5}$}] {};    
    \node (i) at (0,4.9) [vertex, label=above:\tiny{$y$}] {};
    \node (k) at (-1,.7) [vertex] {};
    \node (l) at (-1,1.4) [vertex] {};
    \node (m) at (-.5,2.1) [vertex] {};    
    \node (n) at (-1.5,2.1) [vertex] {};
    \node (o) at (-1,2.8) [vertex] {};
    \node (p) at (-1,3.5) [vertex] {};
    \node (q) at (-1,4.2) [vertex] {};
    \node (1) at (0,-1) {(ii)};
    
    \path [draw=red, ->] (a) edge node[right] {\tiny{$n$}} (b);
    \path [draw=red, ->](b) edge node[right] {\tiny{$n-1$}} (c);
    \path [draw=red, ->](c) edge node[right] {\tiny{$n-1$}} (d);
    \path [draw=red, ->](d) edge node[right] {\tiny{$n-1$}} (f);
    \path [->](d) edge node[left] {\tiny{$n$}} (e);
    \path [->](e) edge node[left] {\tiny{$n-1$}} (g);
    \path [draw=red, ->](f) edge node[right] {\tiny{$n$}} (g);
    \path [draw=red, ->](g) edge node[right] {\tiny{$n$}} (h); 
    \path [draw=red, ->](h) edge node[right] {\tiny{$n-1$}} (i); 
    \path [->](a) edge node[left] {\tiny{$n-1$}} (k); 
    \path [->](k) edge node[left] {\tiny{$n$}} (l); 
    \path [->](l) edge node[right] {\tiny{$n$}} (m); 
    \path [->](l) edge node[left] {\tiny{$n-1$}} (n); 
    \path [->](m) edge node[right] {\tiny{$n-1$}} (o); 
    \path [->](n) edge node[left] {\tiny{$n$}} (o); 
    \path [->](o) edge node[left] {\tiny{$n-1$}} (p);
    \path [->](p) edge node[left] {\tiny{$n-1$}} (q);
    \path [->](q) edge node[left] {\tiny{$n$}} (i);
\end{tikzpicture}

\end{figure}
\begin{figure}[h]
\centering
	\begin{tikzpicture}
[auto,
vertex/.style={circle,draw=black!100,fill=black!,thick,inner sep=0pt,minimum size=1mm}]
    \node (a) at (0,0) [vertex, label=below:\tiny{$x$}] {};
    \node (b) at (-1,1) [vertex] {};
    \node (c) at (-1,2) [vertex] {};
    \node (d) at (-1,3) [vertex] {};
    \node (e) at (-1,4) [vertex] {};
    \node (f) at (0,5) [vertex, label=above:\tiny{$y$}] {};
    \node (g) at (1,1) [vertex, label=right:\tiny{$u_{0}$}] {};
    \node (h) at (1,2) [vertex, label=right:\tiny{$u_{1}$}] {};    
    \node (i) at (1.5,3) [vertex, label=right:\tiny{$u_{2}$}] {};
    \node (j) at (.5,3) [vertex] {};
    \node (k) at (1,4) [vertex, label=right:\tiny{$u_{3}$}] {};
    \node (5) at (3, 2.5) {or};
    \node (1) at (0,-1) {(iii)};
    
    \path [->] (a) edge node[left] {\tiny{$n-1$}} (b);
    \path [->](b) edge node[left] {\tiny{$n$}} (c);
    \path [->](c) edge node[left] {\tiny{$n$}} (d);
    \path [->](d) edge node[left] {\tiny{$n$}} (e);
    \path [->](e) edge node[left] {\tiny{$n-1$}} (f);
    \path [draw=red, ->](a) edge node[right] {\tiny{$n$}} (g);
    \path [draw=red, ->](g) edge node[right] {\tiny{$n-1$}} (h);
    \path [draw=red, ->](h) edge node[right] {\tiny{$n-1$}} (i); 
    \path [->](h) edge node[left] {\tiny{$n$}} (j); 
    \path [draw=red, ->](i) edge node[right] {\tiny{$n$}} (k); 
    \path [->](j) edge node[left] {\tiny{$n-1$}} (k); 
    \path [draw=red, ->](k) edge node[right] {\tiny{$n$}} (f); 
\end{tikzpicture}
\qquad
\begin{tikzpicture}
[auto,
vertex/.style={circle,draw=black!100,fill=black!,thick,inner sep=0pt,minimum size=1mm}]
    \node (a) at (0,0) [vertex, label=below:\tiny{$x$}] {};
    \node (b) at (-1,.7) [vertex] {};
    \node (c) at (-1,1.4) [vertex] {};
    \node (d) at (-1,2.1) [vertex] {};
    \node (e) at (-1.5,2.8) [vertex] {};
    \node (f) at (-.5,2.8) [vertex] {};
    \node (g) at (-1,3.5) [vertex] {};
    \node (h) at (-1,4.2) [vertex] {};    
    \node (i) at (0,4.9) [vertex, label=above:\tiny{$y$}] {};
    \node (k) at (1,.7) [vertex, label=right:\tiny{$u_{0}$}] {};
    \node (l) at (1,1.4) [vertex, label=right:\tiny{$u_{1}$}] {};
    \node (m) at (1.5,2.1) [vertex, label=right:\tiny{$u_{2}$}] {};    
    \node (n) at (.5,2.1) [vertex] {};
    \node (o) at (1,2.8) [vertex, label=right:\tiny{$u_{3}$}] {};
    \node (p) at (1,3.5) [vertex, label=right:\tiny{$u_{4}$}] {};
    \node (q) at (1,4.2) [vertex, label=right:\tiny{$u_{5}$}] {};
    \node (1) at (0,-1) {(iv)};
    
    \path [->] (a) edge node[left] {\tiny{$n-1$}} (b);
    \path [->](b) edge node[left] {\tiny{$n$}} (c);
    \path [->](c) edge node[left] {\tiny{$n$}} (d);
    \path [->](d) edge node[right] {\tiny{$n$}} (f);
    \path [->](d) edge node[left] {\tiny{$n-1$}} (e);
    \path [->](e) edge node[left] {\tiny{$n$}} (g);
    \path [->](f) edge node[right] {\tiny{$n-1$}} (g);
    \path [->](g) edge node[left] {\tiny{$n-1$}} (h); 
    \path [->](h) edge node[left] {\tiny{$n$}} (i); 
    \path [draw=red, ->](a) edge node[right] {\tiny{$n$}} (k); 
    \path [draw=red, ->](k) edge node[right] {\tiny{$n-1$}} (l); 
    \path [draw=red, ->](l) edge node[right] {\tiny{$n-1$}} (m); 
    \path [->](l) edge node[left] {\tiny{$n$}} (n); 
    \path [draw=red, ->](m) edge node[right] {\tiny{$n$}} (o); 
    \path [->](n) edge node[left] {\tiny{$n-1$}} (o); 
    \path [draw=red, ->](o) edge node[right] {\tiny{$n$}} (p);
    \path [draw=red, ->](p) edge node[right] {\tiny{$n$}} (q);
    \path [draw=red, ->](q) edge node[right] {\tiny{$n-1$}} (i);
\end{tikzpicture}
\caption{Additional minimal skipped intervals in doubly laced case}\label{Sternberg relations}
\end{figure}
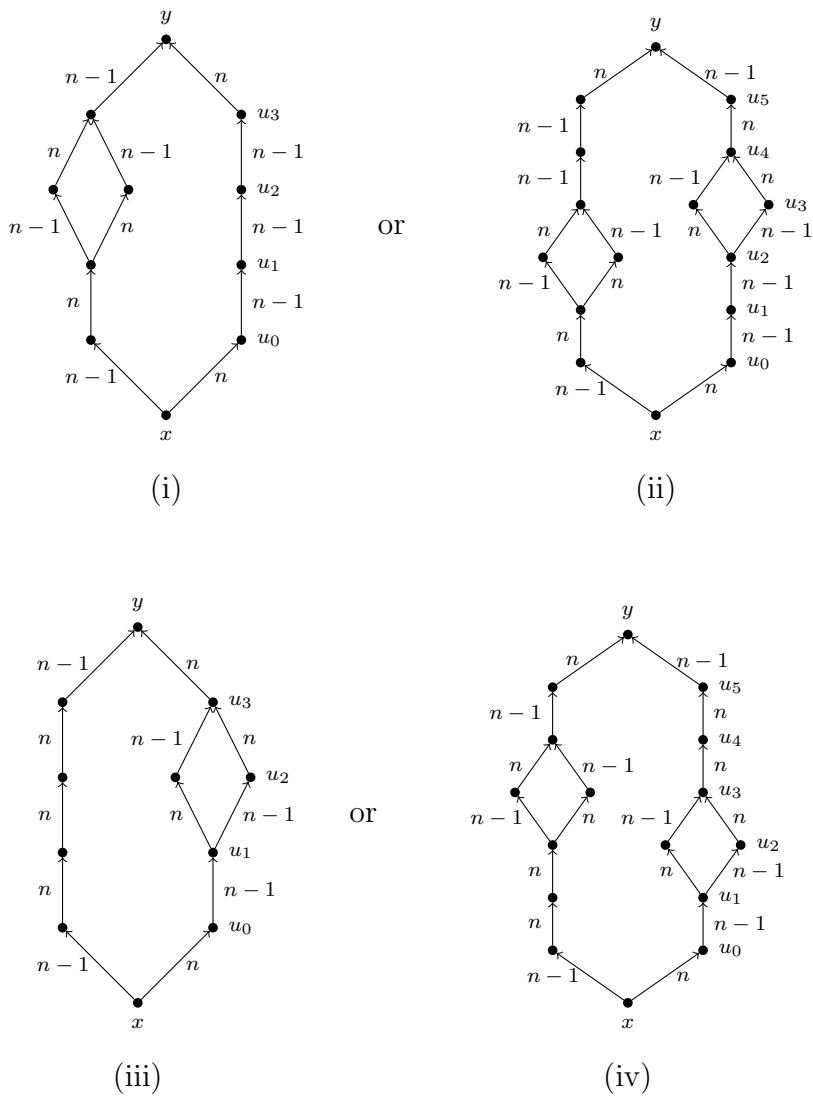

\begin{rmk}
Note that unlike in the simply laced case, the chain within the Sternberg relations that is a candidate to be a part of a fully covered saturated chain is not always lexicographically last. This is due to the degree two Stembridge relations sitting inside the degree five and degree seven Sternberg relations.
\end{rmk}

\begin{definition}
Let $x$ be a vertex along a saturated chain $C$ in $[u,v]$ such that there is a minimal skipped interval for the interval system of $C$ involving the crystal operators $f_{i}$ and $f_{l}$ that begins at $x$, where the edge created by applying $f_{i}$ to $x$ is along $C$. Let $I'$ be the set of indices of the crystal operators that need to be applied along $C$ from $f_{i}(x)$ to $v$. We say that $f_{j}$ is the \textit{maximal operator for $f_{i}$ at $x$} if 
\begin{equation*}
j = \max \{k | k \in I' \text{ and } k < i \}.
\end{equation*}
\end{definition}

\begin{rmk}
This is well defined since there is a finite choice of crystal operators used in a crystal of finite classical type and we can always determine which crystal operators will be used along any saturated chain by Lemma \ref{saturated chain same labels}. Also note that $j$ need not equal $l$.
\end{rmk}

\begin{definition}
We define a saturated chain $C$ to be \textit{greedily maximal} if given any $x$ along $C$ where the edge created by applying $f_{i}$ to $x$ along $C$ is the start of a minimal skipped interval involving the crystal operators $f_{i}$ and $f_{j}$, we have that $f_{j}$ is the maximal operator for $f_{i}$ at $x$.
\end{definition}

\begin{rmk}
Since we are assuming every minimal skipped interval arises from a Stembridge or Sternberg relation, each minimal skipped interval involves exactly two crystal operators. Therefore, this definition is well-defined.
\end{rmk}

In order to prove our main result connecting the M\"obius function of an interval $[u,v]$ with relations among crystal operators within this interval, we first prove a series of lemmas. We begin by proving the following lemma for crystals of highest weight representations of all finite types.

\begin{lemma}\label{max first}
Let $[u,v] \subseteq \mathcal{B}$ be a Stembridge or Stembridge and Sternberg only interval, for $\mathcal{B}$ the crystal of a highest weight representation of finite type. Let $j = \max \{ k | f_{k}$ is applied along any saturated chain from $u \text{ to } v \}$, then $f_{j}$ must be the first operator applied along a fully covered saturated chain.\end{lemma}
\begin{proof}
Suppose by way of contradiction that there is a fully covered saturated chain, $C$, from $u$ to $v$ such that $f_{j}$ is not the first operator applied along $C$. Consider the first occurrence of the crystal operator $f_{j}$ as we proceed upward along $C$ from $u$ towards $v$, namely the first edge colored $j$. By definition of $j$, the label $k$ on the edge immediately preceding the edge colored $j$ on $C$ is such that $k < j$. Since all Stembridge and Sternberg relations involve exactly two crystal operators and all minimal skipped intervals in $[u,v]$ arise from Stembridge or Sternberg relations, the rank corresponding to the vertex labeled $x$ below on the fully covered saturated chain $C$ will not be covered by any minimal skipped intervals, as we justify next.
\begin{figure}[h]
\centering
\begin{tikzpicture}
[auto,
vertex/.style={circle,draw=black!100,fill=black!,thick,inner sep=0pt,minimum size=1.2mm}]
    \node (C) at (-3.5,0) {$C:$};
    \node (u) at (-3,0) [vertex, label=below:\tiny{$u$}] {};
    \node (ellipsis 1) at (-1.5,0) {$\cdots$};
    \node (x) at (0,0) [vertex] {};
    \node (y) at (1,0)[vertex, label=below:\tiny{$x$}] {};
    \node (z) at (2,0) [vertex] {};
    \node (ellipsis 2) at (3.5,0) {$\cdots$};
    \node (v) at (5,0) [vertex, label=below:\tiny{$v$}] {};
    
    \path [thick, draw=black, ->] (u) edge node[right] {} (ellipsis 1);
    \path [thick, draw=black, ->] (ellipsis 1) edge node[above] {} (x);
    \path [thick, draw=black, ->] (x) edge node[above] {\tiny{$k$}} (y);
    \path [thick, draw=black, ->](y) edge node[above] {\tiny{$j$}} (z);
    \path [thick, draw=black, ->] (z) edge node[above] {} (ellipsis 2);
    \path [thick, draw=black, ->] (ellipsis 2) edge node[above] {} (v);

\end{tikzpicture}
\end{figure}
If the rank corresponding to the vertex labeled $x$ were covered by some minimal skipped interval, the corresponding Stembridge or Sternberg relation must involve the crystal operators $f_{k}$ and $f_{j}$. However, since $k < j$, this piece of the Stembridge or Sternberg relation along $C$ will be lexicographically earlier than the piece with edge label sequence $(j,k)$. Hence, we will not have a minimal skipped interval covering the rank corresponding to the vertex $x$. This contradicts the saturated chain $C$ being fully covered.
\end{proof}

The interval systems for crystals of simply laced types behave differently than those for doubly laced types. We first focus on results for simply laced types and then generalize to the doubly laced case.

\begin{lemma}\label{nonoverlapping}
Let $[u,v] \subseteq \mathcal{B}$ where $\mathcal{B}$ is the crystal of a highest weight representation of simply laced type. Assume $[u,v]$ is a Stembridge only interval. In this case, the interval system of any fully covered saturated chain in $[u,v]$ is non-overlapping i.e. no truncation algorithm is necessary. 
\end{lemma}
\begin{proof}
Let $C$ be a fully covered saturated chain from $u$ to $v$ and let $I$ be the interval system for $C$. Any minimal skipped interval in $I$ is of the form seen in Figure \ref{Stem type A}. The first type of minimal skipped interval coming from the degree two Stembridge relation covers exactly one rank. Therefore, any minimal skipped interval arising from this relation cannot overlap with another minimal skipped interval. Hence, we restrict our attention to minimal skipped intervals that arise from the degree four Stembridge relation $f_{i}f_{i+1}^{2}f_{i} = f_{i+1}f_{i}^{2}f_{i+1}$ (or potentially $f_{n-2}f_{n}^{2}f_{n-2} = f_{n}f_{n-2}^{2}f_{n}$ in type $D_{n}$). 

Suppose we have some vertex $x \in C$ such that there is a minimal skipped interval for the interval system of $C$ beginning at $x$ coming from a degree four Stembridge relation. If there exists another minimal skipped interval that overlaps with the one arising from this Stembridge relation, then using the notation from Figure \ref{Stem type A}, it must either begin at the vertex $u_{0}$ or the vertex $u_{1}$. Since $[u,v]$ is a Stembridge only interval, if we have a minimal skipped interval beginning at $u_{0}$ or $u_{1}$, it must come from a degree two or degree four Stembridge relation. In fact, it must come from a degree four Stembridge relation. If not, the minimal skipped interval beginning at $x$ for the interval system of $C$ that arises from the degree four Stembridge relation would not be minimal. 

However, we cannot have a minimal skipped interval beginning at $u_{0}$ because the lexicographically last chain in a degree four Stembridge relation does not have an edge label sequence beginning with $i, i, i+1$. We also cannot have a minimal skipped interval beginning at $u_{1}$ since we have $f_{i}$ being applied before $f_{i+1}$ (or potentially $f_{n-2}$ being applied before $f_{n}$ in type $D_{n}$), and therefore we would only see the lexicographically earlier piece of a Stembridge relation on $C$. As a result, this will not give rise to a minimal skipped interval. Therefore, the interval system of $C$ will be non-overlapping.  
\end{proof}

The following two lemmas give restrictions on the structure of fully covered saturated chains in the simply laced case. These are necessary to prove that if there is a a fully covered saturated chain in a given interval, then this chain is unique.

\begin{lemma}\label{greedily maximal}
Let $\mathcal{B}$ be the crystal of a highest weight representation of type $A_{n}$ and $[u,v] \subseteq \mathcal{B}$ be a Stembridge only interval, then any fully covered saturated chain in $[u,v]$ is greedily maximal.
\end{lemma}
\begin{proof}

Let $x$ be a vertex along a fully covered saturated chain $C$ such that $x$ is the last rank covered by some minimal skipped interval in the interval system for $C$. Since $C$ is fully covered and there is no overlap among minimal skipped intervals in the simply laced case, $x$ must be the start of a new minimal skipped interval for the interval system of $C$. Suppose the first edge along $C$ in this minimal skipped interval is labeled $i$, i.e. comes from the application of the crystal operator $f_{i}$. Let $j$ be the index such that $f_{j}$ is the maximal operator for $f_{i}$ at $x$. Assume by way of contradiction that the minimal skipped interval upward from $x$ involves $f_{i}$ and $f_{k}$ where $k \neq j$. Since $f_{k}$ is not the maximal operator for $f_{i}$ at $x$, we know that $k <j$.

We note that since $i > j > k$, we cannot have $k = i - 1$. This implies the minimal skipped interval involving $f_{i}$ and $f_{k}$ arises from a degree two Stembridge relation, $f_{k}f_{i}(x) = f_{i}f_{k}(x)$. Therefore, the next time there is an edge colored $j$ upward from $x$ to $v$ along $C$, the edge below it on $C$ will have label strictly less than $j$ by definition of maximal operator. This contradicts $C$ being fully covered via the same argument as the proof of Lemma \ref{max first}.
\end{proof}

\begin{rmk}
The idea for crystals of type $D_{n}$ will be similar but will require some extra care. We need to take into account that for crystals of type $D_{n}$ the Stembridge relations that can occur among the crystal operators $f_{n-2}, f_{n-1}$ and $f_{n}$ are different than those that occur in type $A_{n}$. Namely, $f_{n-1}$ and $f_{n}$ can only be involved in a degree two Stembridge relation, while it is possible to have a degree four Stembridge relation involving $f_{n-2}$ and $f_{n}$. This is the content of Lemma \ref{type D}. See Example \ref{type D example} for an illustration of this.
\end{rmk}

\begin{definition}
Let $\mathcal{B}$ be the crystal of a highest weight representation of type $D_{n}$. Let $[u,v] \subseteq \mathcal{B}$ be a Stembridge only interval. Let $x$ be a vertex on a fully covered saturated chain $C$ from $u$ to $v$. We say that $x$ is an \textit{(n,n-2)-special vertex} if there is an $n$-edge upward from $x$ along $C$ which is the start of a minimal skipped interval for the interval system of $C$ and the number of times $f_{n}$ needs to be applied along $C$ from $f_{n}(x)$ to $v$ is nonzero.
\end{definition}


\begin{lemma}\label{type D}
Suppose that $[u,v] \subseteq \mathcal{B}$ is a Stembridge only interval, for $\mathcal{B}$ a crystal of a highest weight representation of type $D_{n}$. Let $C$ be a fully covered saturated chain from $u$ to $v$. For any $(n,n-2)$-special vertex $x$ along $C$, $f_{n-2}$ is the maximal operator for $f_{n}$ at $x$. Under this condition, $C$ is greedily maximal.
\end{lemma}
\begin{proof}
Recall that in crystals of type $D_{n}$, $f_{n}$ and $f_{n-1}$ commute whenever there exists a $y \in [u,v]$ such that $f_{n}(y) \neq 0$ and $f_{n-1}(y) \neq 0$. Let $C$ be a fully covered saturated chain from $u$ to $v$ and $x \in C$ be an $(n,n-2)$-special vertex. If the minimal skipped interval upward from $x$ involves $f_{n}$ and $f_{i}$ where $i \neq n-2$, then there will be an uncovered rank as we travel upwards from $x$ towards $v$ along $C$, which we justify next.

The claim holds because the minimal skipped interval involving $f_{n}$ and $f_{i}$ beginning at $x$ in $C$ will come from a degree two Stembridge relation. Consider the next edge labeled $n$ proceeding upwards along $C$. Since $n$ is the largest possible edge label occurring on saturated chains from $u$ to $v$, the edge in $C$ below the edge colored $n$ will have label $k$ for some $k \in [n]$ where $k < n$. By the nature of Stembridge relations, the rank corresponding to the vertex between the $k$ edge and the $n$ edge must be uncovered as seen in the proof of Lemma \ref{max first}. 

Therefore, the only way to have $C$ be a fully covered saturated chain is if the maximal operator for $f_{n}$ at $x$ is $f_{n-2}$. This is because if the minimal skipped interval for $C$ beginning at $x$ comes from a degree four Stembridge relation involving $f_{n}$ and $f_{n-2}$, then the next time there is a vertex on $C$ such that $f_{n}(x)$ is also along $C$, it is the start of a new minimal skipped interval and the rank corresponding to this vertex is contained in a previous minimal skipped interval. Hence, in order to potentially have a fully covered saturated chain in this case, $f_{n-2}$ must play the role of the maximal operator for $f_{n}$. If we view $f_{n-2}$ as the maximal operator for $f_{n}$ when we have an $(n,n-2)$-special vertex, then the proof of a fully covered saturated chain $C$ being greedily maximal is analogous to the type $A_{n}$ case seen in Lemma \ref{greedily maximal}.
\end{proof}

\begin{ex}\ytableausetup{boxsize=1.1em}\label{type D example}
Consider the type $D_{3}$ crystal $\mathcal{B}_{2,1,1}$ and the interval $[u,v]$ shown below where
 
\begin{equation*} u = \fontsize{6pt}{7} \begin{ytableau} \tiny{1} & 2 \\ \overline{3} \\ 3 \end{ytableau}, \quad v = \fontsize{6pt}{7} \begin{ytableau} 2 & \overline{2} \\ \overline{3} \\ \overline{1} \end{ytableau}.
\end{equation*}
Here, we have $[u,v]$ is a Stembridge only interval and $\mu(u,v) = -1$. To get from $u$ to $v$, the crystal operator $f_{3}$ must be applied more than once. We consider the saturated chain $C$ with label sequence $(3,1,1,3,2)$. This saturated chain is fully covered by two minimal skipped intervals. The first comes the edges beginning at $u$ with label sequence $(1,3,3,1)$. The second minimal skipped interval begins at the fourth vertex along $C$ and comes from the edges with label sequence $(2,3)$. Here, we have that $u$ is a $(3,1)$-special vertex and $f_{1}$ plays the role of maximal operator for $f_{3}$ at $u$ instead of $f_{2}$. We note that $C$ is not the lexicographically last chain in this interval.

\begin{figure}[h]
\fontsize{4.8pt}{6.7}
\centering
	\begin{tikzpicture}
    \node (a) at (0,0) {$\begin{ytableau} 1 & 2 \\ \overline{3}  \\ 3 \end{ytableau}$};
    \node (b) at (-2,2)  {$\begin{ytableau} 2 & 2 \\ \overline{3}  \\ 3 \end{ytableau}$};
    \node (c) at (0,2)  {$\begin{ytableau} 1 & 3 \\ \overline{3}  \\ 3 \end{ytableau}$};
    \node (d) at (2,2)  {$\begin{ytableau} 1 & 2 \\ \overline{3}  \\ \overline{2} \end{ytableau}$};
    \node (e) at (-4,4)  {$\begin{ytableau} 2 & 3 \\ \overline{3}  \\ 3 \end{ytableau}$};
    \node (f) at (-2,4) {$\begin{ytableau} 2 & \overline{3} \\ \overline{3}  \\ 3 \end{ytableau}$};
    \node (g) at (0,4) {$\begin{ytableau} 1 & 3 \\ \overline{3}  \\ \overline{2} \end{ytableau}$};   
    \node (h) at (2,4) {$\begin{ytableau} 2 & 2 \\ \overline{3}  \\ \overline{2} \end{ytableau}$};    
    \node (i) at (-4,6) {$\begin{ytableau} 2 & \overline{2} \\ \overline{3}  \\ 3 \end{ytableau}$};
    \node (k) at (-2,6) {$\begin{ytableau} 2 & \overline{3} \\ \overline{3}  \\ \overline{2} \end{ytableau}$};
    \node (l) at (0,6) {$\begin{ytableau} 2 & 3 \\ \overline{3}  \\ \overline{2} \end{ytableau}$};
    \node (m) at (2,6) {$\begin{ytableau} 2 & 2 \\ \overline{3}  \\ \overline{1} \end{ytableau}$};    
    \node (n) at (-2,8) {$\begin{ytableau} 2 & \overline{2} \\ \overline{3}  \\ \overline{2} \end{ytableau}$}; 
    \node (o) at (0,8) {$\begin{ytableau} 2 & \overline{3} \\ \overline{3}  \\ \overline{1} \end{ytableau}$};
    \node (p) at (2,8) {$\begin{ytableau} 2 & 3 \\ \overline{3}  \\ \overline{1} \end{ytableau}$};
    \node (q) at (0,10) {$\begin{ytableau} 2 & \overline{2} \\ \overline{3}  \\ \overline{1} \end{ytableau}$};

    \path [draw=blue, ->] (a) edge node[left] {$1$} (b);
    \path [draw=red, ->](a) edge node[right] {$2$} (c);
    \path [draw=green, ->](a) edge node[right] {$3$} (d);
    \path [draw=red, ->](b) edge node[left] {$2$} (e);
    \path [draw=green, ->](b) edge node[right] {$3$} (f);
    \path [draw=blue, ->](c) edge node[left] {$1$} (e);
    \path [draw=green, ->](c) edge node[right] {$3$} (g);
    \path [draw=red, ->](d) edge node[right] {$2$} (g); 
    \path [draw=blue, ->](d) edge node[right] {$1$} (h); 
    \path [draw=green, ->](e) edge node[left] {$3$} (i); 
    \path [draw=red, ->](f) edge node[right] {$2$} (i); 
    \path [draw=green, ->](f) edge node[right] {$3$} (k); 
    \path [draw=blue, ->](g) edge node[left] {$1$} (l); 
    \path [draw=red, ->](h) edge node[right] {$2$} (l); 
    \path [draw=blue, ->](h) edge node[right] {$1$} (m); 
    \path [draw=green, ->](i) edge node[left] {$3$} (n);
    \path [draw=red, ->](k) edge node[right] {$2$} (n);
    \path [draw= blue, ->](k) edge node[right] {$1$} (o);
    \path [draw=blue, ->] (l) edge node[left] {$1$} (p);
    \path [draw=green, ->] (m) edge node[right] {$3$} (o);
    \path [draw=red, ->] (m) edge node[right] {$2$} (p);
    \path [draw=blue, ->] (n) edge node[left] {$1$} (q);
    \path [draw=red, ->] (o) edge node[left] {$2$} (q);
    \path [draw=green, ->] (p) edge node[right] {$3$} (q);
\end{tikzpicture}

\end{figure}

\end{ex}

\begin{rmk}
While Lemma \ref{max first} and Lemma \ref{nonoverlapping} hold for all simply laced and doubly laced finite Kac Moody algebras, the next two Lemmas are not able to be generalized to include types $E_{6}, E_{7}$ and $E_{8}$. For these types, we are able to have degree four Stembridge relations occurring among the crystal operators $f_{4}$ and $f_{3}$ as well as $f_{4}$ and $f_{2}$. This implies that we may not have a unique fully covered saturated chain as the greedily maximal chain may not be the only fully covered one. 
\end{rmk}

Lemma \ref{type D} gives the type $D$ analogue to Lemma \ref{greedily maximal} regarding fully covered saturated chains being greedily maximal. We now give a description of how to find the unique fully covered saturated chain in crystals of types $A_{n}$ and $D_{n}$ when it exists. 

\begin{thm}\label{one sat chain type A}
Suppose we have an interval $[u,v] \subseteq \mathcal{B}$ that is Stembridge only, for $\mathcal{B}$ the crystal of a highest weight representation of type $A_{n}$ or $D_{n}$. In this case, there is at most one fully covered saturated chain in $[u,v]$.
\end{thm}
\begin{proof}
If we have a fully covered saturated chain in $[u,v]$, then it is greedily maximal as seen in Lemmas \ref{greedily maximal} and \ref{type D}. 
We now describe how to find the unique greedily maximal chain in a given interval when it exists. This will tell us that in any interval, there is at most one fully covered saturated chain.

From Lemma \ref{max first}, we know that in order to have a fully covered saturated chain, the chain must start with the application of the crystal operator $f_{k}$ where \[k = \max \{i | f_{i} \text{ applied to get from } u \text{ to } v \}.\]
Moreover, this says that if $f_{k}(u) = 0$, then there is no fully covered saturated chain in $[u,v]$. Assuming now that $f_{k}(u) \neq 0$, we next need to apply $f_{j}$ where $f_{j}$ is the maximal operator for $f_{k}$ at $u$. The beginning of our potential fully covered saturated chain is now $u \lessdot f_{k}(u) \lessdot f_{j}f_{k}(u)$. In order for the rank corresponding to the vertex $f_{k}(u)$ to be covered by a minimal skipped interval, we need that the chain $u \lessdot f_{k}(u) \lessdot f_{j}f_{k}(u)$ is contained within a Stembridge relation. This will only be the case if $f_{j}(u) \neq 0$. If $f_{j}(u) = 0$, then we will not have a fully covered saturated chain in this interval because the rank corresponding to the vertex $f_{k}(u)$ will be uncovered. If $f_{j}(u) \neq 0$, then we travel up the lexicographically later chain in the Stembridge relation upward from $u$ involving $f_{k}$ and $f_{j}$. 

Once we hit the last rank that this minimal skipped interval covers, we repeat the process above. More specifically, this minimal skipped interval described above either ends with the application of $f_{k}$ (in the case where we have a degree four Stembridge relation between $f_{k}$ and $f_{j}$) or $f_{j}$ (in the case where we have a degree two Stembridge relation between $f_{k}$ and $f_{j}$). We then find the maximal operator for either $f_{k}$ or $f_{j}$ depending on if we had a degree two or degree four Stembridge relation and see if this portion of our chain is contained in another Stembridge relation. If not, there is no fully covered saturated chain in this interval. We continue this process until we reach $v$. If there is a saturated chain from $u$ to $v$ such that each minimal skipped interval involves crystal operators $f_{i}$ and $f_{j}$ with $i > j$ such that $f_{j}$ is the maximal operator for $f_{i}$, then we have a fully covered saturated chain. Note that since we chose maximal operators at each step, this chain is uniquely described.
\end{proof}

\begin{cor}\label{simply laced mobius}
For an interval as above, we have $\mu(u,v) \in \{-1,0,1\}.$
\end{cor}
\begin{proof}
This follows from the correspondence of the reduced Euler characteristic of the order complex of an open interval with the M\"obius function of the interval. More specifically, we have the following:
\[\mu_{P}(u,v) = \tilde{\chi}(\Delta(u,v)) = \tilde{\chi}(\Delta^{M}(u,v)),\]
where $\Delta^{M}(u,v)$ is the CW-complex obtained from the discrete Morse function. Since there is at most one fully covered saturated chain, we have at most one critical cell other than the basepoint. In this case, the cell complex is homotopy equivalent to a sphere of dimension of the critical cell. Hence, the reduced Euler characteristic will be $\pm 1$ if there is a fully covered saturated chain and $0$ otherwise. 
\end{proof}

\begin{rmk}
The converse of Corollary \ref{simply laced mobius} is not true. There do exist intervals $[u,v]$ in crystals of highest weight representations of simply laced type such that $\mu(u,v) \in \{-1,0,1\}$ where there are relations among crystal operators that are not implied by Stembridge relations.
\end{rmk}

In practice, we use the contrapositive of Corollay \ref{simply laced mobius} to search for new relations among crystal operators as will be seen for the doubly laced case in Section 5. We state it here as a corollary.

\begin{cor} \label{main result simply laced}
Let $[u,v] \in \mathcal{B}$, for $\mathcal{B}$ the crystal of a highest weight representation of type $A_{n}$ or $D_{n}$. If $\mu(u,v) \notin \{-1,0,1\}$, then there exists a relation among crystal operators that is not implied by Stembridge relations.
\end{cor}


We now move on to crystals of types $B_{n}$ and $C_{n}$. 
For crystals of type $B_{n}$ and $C_{n}$ it is possible to have a saturated chain whose interval system is overlapping. If this occurs, we use the truncation algorithm to see if we have a fully covered saturated chain. 

\begin{lemma}\label{degree 5 overlap}
Suppose $\mathcal{B}$ is the crystal of a highest weight representation of type $B_{n}$ or $C_{n}$. Let $[u,v] \subseteq \mathcal{B}$ such that $[u,v]$ is a Stembridge and Sternberg only interval. Then there can be no overlap among minimal skipped intervals if no minimal skipped intervals arise from a degree seven Sternberg relation.
\end{lemma}
\begin{proof}
The minimal skipped intervals arising from degree two and degree four Stembridge relations remain non-overlapping in the doubly laced case by the same argument used in Lemma \ref{nonoverlapping}. Therefore, we only need to show that the interval system of some saturated chain $C$ is non-overlapping when there is a minimal skipped interval that arises from a degree five Sternberg relation. The argument is analogous to that of the degree four Stembridge case. 

First, suppose we have a minimal skipped interval coming from (i) in Figure \ref{Sternberg relations}. In this case, the minimal skipped interval covers the ranks corresponding to the vertices $\{u_{0}, u_{1}, u_{2}, u_{3}\}$. Therefore, we need to rule out the possibility of a minimal skipped interval beginning at $u_{0}, u_{1},$ or $u_{2}$. In all of these cases, the minimal skipped interval would involve the crystal operators $f_{n-1}$ and $f_{n}$. However, no saturated chains within a Stembridge or Sternberg relation begins with multiple applications of $f_{n-1}$. Hence, there are no minimal skipped intervals beginning at $u_{0}$ or $u_{1}$. In addition, there is no minimal skipped interval beginning at $u_{2}$ because any Stembridge or Sternberg relation upward from $u_{2}$ would involve $f_{n-1}$ and $f_{n}$. However, since we have $f_{n-1}$ being applied before $f_{n}$, we would only see the lexicographically earlier chain within the Stembridge or Sternberg relation. Therefore, this would not give rise to a minimal skipped interval.

Next, we consider case (iii) from Figure \ref{Sternberg relations}. Again, the only possibilities would be for a minimal skipped interval to begin from $u_{0}, u_{1},$ or $u_{2}$ and it would need to involve the crystal operators $f_{n-1}$ and $f_{n}$. But as before, no saturated chain within a Stembridge or Sternberg relation begins with the repeated application of a single crystal operator so we cannot have a new minimal skipped interval beginning at $u_{0}$ or $u_{2}$. Also, any chain in a Stembridge or Sternberg relation involving $f_{n-1}$ and $f_{n}$ beginning with $f_{n-1}$ will be the lexicographically earlier chain within that Stembridge or Sternberg relation. As a result, there will be no overlap among minimal skipped intervals beginning from a degree five Sternberg relation.
\end{proof}

\begin{lemma}\label{overlapping MSI}
Suppose $\mathcal{B}$ is the crystal of a highest weight representation of type $B_{n}$ or $C_{n}$. Let $[u,v] \subseteq \mathcal{B}$ such that $[u,v]$ is a Stembridge and Sternberg only interval. Let $C$ be a saturated chain from $u$ to $v$ such that it's interval system covers all ranks but with overlap among minimal skipped intervals, then $C$ remains fully covered after the truncation algorithm.
\end{lemma}
\begin{proof}
From Lemma \ref{nonoverlapping} and Lemma \ref{degree 5 overlap}, we know that there is no overlap between minimal skipped intervals that arise from Stembridge relations and degree five Sternberg relations. Therefore, we restrict our attention to fully covered saturated chains that have a minimal skipped interval arising from a degree seven Sternberg relation. Let $C$ be one such fully covered saturated chain. 

Suppose the minimal skipped interval for $C$ coming from the degree seven Sternberg relation is of the form seen in Figure \ref{Sternberg relations} (ii) above. Assume this minimal skipped interval begins at some vertex $x \in C$. If the number of times $f_{n}$ needs to be applied to get from $x$ to $v$ is greater than three, then in order for $C$ to be fully covered, we must have overlap among minimal skipped intervals. Similarly, if the minimal skipped interval comes from the degree seven relation seen in Figure \ref{Sternberg relations} (iv) and the number of times $f_{n}$ needs to be applied to get from $x$ to $v$ is greater than four, then there will be overlap among minimal skipped intervals. This is because in either case the piece of the degree seven Sternberg relation chain along our saturated chain $C$ ends with the crystal operator $f_{n-1}$. If $f_{n}$ still needs to be applied along $C$ to reach $v$, the rank corresponding to the first vertex $y$ in $C$ such that $f_{n}(y)$ is also along $C$ will not be contained in a minimal skipped interval if there is no overlap. This is because the edge along $C$ below $y$ will have label $i$ for some $i \in [n]$ where $i < n$. The rank corresponding to the vertex $y$ is uncovered by the same argument seen in Lemma \ref{max first}. To remedy this, we need to be able to have another minimal skipped interval begin with the application of $f_{n}$. This can only happen if there is overlap.

In either case, the overlap among minimal skipped intervals will include only the rank marked by the vertex $u_{5}$ from Figure \ref{Sternberg relations}. The proof of why this is the case is analogous to that seen in Lemma \ref{degree 5 overlap}. Since all minimal skipped intervals arise from Stembridge and Sternberg relations, a new interval can only arise off of the degree seven Sternberg if it starts at $u_{4}$. Depending on how many times $f_{n}$ is applied from $u$ to $v$, the minimal skipped interval may be a degree four Stembridge, degree five or degree seven Sternberg relation. Note that it cannot arise from a degree two Stembridge relation. If this were the case, the original degree seven minimal skipped interval would not in fact be minimal.

We claim that if the original interval system $I$ covered all ranks of $C$, then the truncated interval system $J$ does as well. Since all minimal skipped intervals arising from Stembridge relations and degree five Sternberg relations do not overlap with each other, we only need to worry about the ranks involved with the overlap coming from the degree seven Sternberg relation. If there is any overlap, there will be exactly one rank that lies in two minimal skipped intervals as described above. Let $I_{1}$ and $I_{2}$ be the two overlapping intervals in $I$.

Using the notation from Figure \ref{Sternberg relations}, $I_{1}$ will contain the ranks $\{u_{0}, u_{1}, u_{2}, u_{3}, u_{4}, u_{5}\}.$ Then $I_{2}$ will contain the ranks $\{u_{5}, w_{0}, ..., w_{k}\}$ where $k \in \{1,...,5\}$ dependent on if $I_{2}$ arises from a degree four, degree five or degree seven relation.  Applying the truncation algorithm will give us $J_{1} = \{u_{0}, u_{1}, u_{2}, u_{3}, u_{4}, u_{5}\}$ and $J_{2} = \{ w_{0},..., w_{k} \}$. Since this will be the only instance of overlap, $C$ is still fully covered as the rest of the intervals in $I$ remain unchanged. 
\end{proof}

\begin{thm} Given any $u < v$ in a crystal $\mathcal{B}$ of a highest weight representation of type $B_{n}$ or $C_{n}$ such that $[u,v]$ is a Stembridge and Sternberg only interval,  then there is at most one fully covered saturated chain in $[u,v]$.
\end{thm}
 
\begin{proof}
If the minimal skipped intervals for the interval system of $C$ are non-overlapping, then the argument from Theorem \ref{one sat chain type A} applies directly. The only change for doubly laced crystals is the potential for overlap coming from minimal skipped intervals that arise from degree seven Sternberg relations. Hence, we limit ourselves to fully covered saturated chains where there is a minimal skipped interval that arises from a degree seven Sternberg relation. 

Let $C$ be one such chain. Suppose $x$ is a vertex along $C$ such that there is a minimal skipped interval for the interval system of $C$ beginning at $x$ coming from a degree seven Sternberg relation. In this case, we check if there is overlap among minimal skipped intervals as described in Lemma \ref{overlapping MSI}. Recall that this overlap can occur at exactly one place. If this happens, we travel up our saturated chain until the end of the last overlapping minimal skipped interval. From there, we once again look for the maximal operator as in the proof of Theorem \ref{one sat chain type A}. At each step, we have exactly one choice so this fully covered saturated chain will be unique, if it exists.
\end{proof}

\begin{cor} \label{doubly laced mobius}
For an interval as above, $\mu(u,v) \in \{-1,0,1\}$.
\end{cor} 
\begin{proof}
The proof is completely analogous to that of Corollary \ref{simply laced mobius}.
\end{proof}

As in the simply laced case, we use the contrapositive of Corollary \ref{doubly laced mobius} to search for new relations among crystal operators. We state this here as a corollary. For an example illustrating this result, see Section \ref{application}.

\begin{cor} \label{main result doubly laced}
Let $[u,v] \in \mathcal{B}$, for $\mathcal{B}$ the crystal of a highest weight representation of type $B_{n}$ or $C_{n}$. If $\mu(u,v) \notin \{-1,0,1\}$, then there exists a relation among crystal operators that is not implied by Stembridge or Sternberg relations.
\end{cor}

\section{Application of Corollary \ref{main result doubly laced}}\label{application}
\begin{figure}[h]
\fontsize{4.8pt}{6.7}
\ytableausetup{boxsize=1.2em}
\centering
\begin{tikzpicture}
    \node (a) at (0,0) {$\begin{ytableau} 1 & 1 & 2 & 3\\ 3 & 3 & \overline{3}  \\ \overline{3} \end{ytableau}$};
    \node (b) at (-2,2)  {$\begin{ytableau} 1 & 1 & 2 & 3\\ 3 & 3 & \overline{3}  \\ \overline{2} \end{ytableau}$};
    \node (c) at (2,2)  {$\begin{ytableau} 1 & 2 & 2 & 3\\ 3 & 3 & \overline{3}  \\ \overline{3} \end{ytableau}$};
    \node (d) at (-3,4)  {$\begin{ytableau} 1 & 1 & 2 & 3\\ 3 & \overline{3} & \overline{3}  \\ \overline{2} \end{ytableau}$};
    \node (e) at (-1,4)  {$\begin{ytableau} 1 & 2 & 2 & 3\\ 3 & 3 & \overline{3}  \\ \overline{2} \end{ytableau}$};
    \node (f) at (2,4) {$\begin{ytableau} 1 & 2 & 3 & 3\\ 3 & 3 & \overline{3}  \\ \overline{3} \end{ytableau}$};
    \node (g) at (-3,6) {$\begin{ytableau} 1 & 1 & 3 & 3\\ 3 & \overline{3} & \overline{3}  \\ \overline{2} \end{ytableau}$};
    \node (h) at (-1,6) {$\begin{ytableau} 1 & 2 & 2 & 3\\ 3 & \overline{3} & \overline{3}  \\ \overline{2} \end{ytableau}$};    
    \node (i) at (1,6) {$\begin{ytableau} 1 & 2 & 3 & \overline{3}\\ 3 & 3 & \overline{3}  \\ \overline{3} \end{ytableau}$};
    \node (k) at (3,6) {$\begin{ytableau} 1 & 2 & 3 & 3\\ 3 & 3 & \overline{3}  \\ \overline{2} \end{ytableau}$};
    \node (l) at (-2,8) {$\begin{ytableau} 1 & 2 & 3 & 3\\ 3 & \overline{3} & \overline{3}  \\ \overline{2} \end{ytableau}$};
    \node (m) at (2,8) {$\begin{ytableau} 1 & 2 & 3 & \overline{3}\\ 3 & 3 & \overline{3}  \\ \overline{2} \end{ytableau}$};    
    \node (n) at (-2,10) {$\begin{ytableau} 1 & 2 & 3 & 3\\ 3 & \overline{3} & \overline{2}  \\ \overline{2} \end{ytableau}$};
    \node (o) at (2,10) {$\begin{ytableau} 1 & 2 & 3 & \overline{3}\\ 3 & \overline{3} & \overline{3}  \\ \overline{2} \end{ytableau}$};
    \node (p) at (0,12) {$\begin{ytableau} 1 & 3 & 3 & \overline{3} \\ \overline{3} & \overline{3} & \overline{1} \\ \overline{2} \end{ytableau}$};

    \path [draw=red, ->] (a) edge node[left] {$2$} (b);
    \path [draw=blue, ->](a) edge node[right] {$1$} (c);
    \path [draw=green, ->](b) edge node[left] {$3$} (d);
    \path [draw=blue, ->](b) edge node[right] {$1$} (e);
    \path [draw=red, ->](c) edge node[right] {$2$} (f);
    \path [draw=red, ->](d) edge node[left] {$2$} (g);
    \path [draw=blue, ->](d) edge node[right] {$1$} (h);
    \path [draw=green, ->](e) edge node[right] {$3$} (h); 
    \path [draw=green, ->](f) edge node[left] {$3$} (i); 
    \path [draw=red, ->](f) edge node[right] {$2$} (k); 
    \path [draw=blue, ->](g) edge node[left] {$1$} (l); 
    \path [draw=red, ->](h) edge node[right] {$2$} (l); 
    \path [draw=red, ->](i) edge node[left] {$2$} (m); 
    \path [draw=green, ->](k) edge node[right] {$3$} (m); 
    \path [draw=red, ->](l) edge node[left] {$2$} (n); 
    \path [draw=green, ->](m) edge node[right] {$3$} (o);
    \path [draw=green, ->](n) edge node[left] {$3$} (p);
    \path [draw= red, ->](o) edge node[right] {$2$} (p);
\end{tikzpicture}

\caption{New relation in type $C_{3}$ crystal $\mathcal{B}_{(4,3,1)}$}
\label{new relation}
\end{figure}
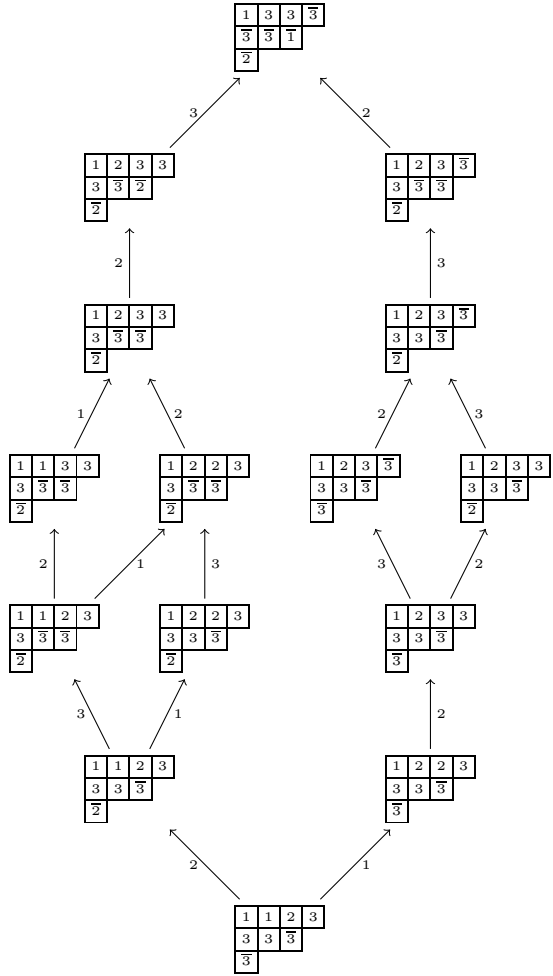

While trying to find new relations among crystal operators is a difficult task, computing the M\"obius function of a given interval is algorithmic and efficient. Specifically, we use SAGE to search for intervals among crystals of finite classical type with M\"obius function not equal to -1, 0, or 1. In general, it is not obvious how to search for new relations among crystal operators. By establishing a relationship between relations among crystal operators and the M\"obius function of intervals within our crystal posets, we have a computational and algorithmic tool to find new relations.

We have found multiple new relations among crystal operators in crystals of type $C_{n}$. We do so by examining intervals where the M\"obius function is not equal to $-1, 0$ or $1$. See Figure \ref{new relation} above for an example of a new relation among crystal operators found in the type $C_{3}$ crystal $\mathcal{B}_{(4,3,1)}$ of shape $\lambda = (4,3,1)$, namely we have $x \in \mathcal{B}_{(4,3,1)}$ such that:
\[ f_{2}f_{3}^{2}f_{2}^{2}f_{1}(x) = f_{2}f{3}f_{2}f_{3}f_{2}f_{1}(x) = f_{3}f_{2}^{2}f_{3}f_{1}f_{2}(x) = f_{3}f_{2}f_{1}f_{2}f_{3}f_{2}(x) = f_{3}f_{2}^{2}f_{1}f_{3}f_{2}(x) \]



\printbibliography





\end{document}